\documentclass[a4paper, 9pt]{article}
\usepackage{graphics}
\usepackage[dvips]{color}

\usepackage{lscape}
\usepackage{latexsym}
\usepackage{amsmath,verbatim}
\usepackage{graphics}
\usepackage{amsthm}
\usepackage{amssymb}
\usepackage[mathscr]{euscript}
\usepackage{yfonts}
\usepackage{makeidx}
\usepackage{stmaryrd}
\usepackage{multicol}
\usepackage{bm}
\usepackage[all]{xy}
\usepackage[stable]{footmisc}

\usepackage{mathtools}

\DeclarePairedDelimiter\ket{\lvert}{\rangle}
\DeclarePairedDelimiterX\braket[2]{\langle}{\rangle}{#1 \delimsize\vert #2}

\numberwithin{equation}{section}
\newtheorem{thm}{Theorem}[section]
\newtheorem{prop}[thm]{Proposition}
\newtheorem{lem}[thm]{Lemma}
\newtheorem{cor}[thm]{Corollary}
\newtheorem{claim}{Claim}{\bf}{\it}

\newtheorem{fthm}{Theorem}{\bf}{\it}
{\bf}{\it}
{\bf}{\it}

\theoremstyle{definition}
\newtheorem{defn}[thm]{Definition}

{\bf}{\rm}

\theoremstyle{remark}

\newtheorem{rem}[thm]{Remark}
{\bf}{\it}

\newtheorem{definition and corollary}[thm]{Definition and Corollary}

{\it}{\rm}
\newcommand{\A}{{\mathbb A}}
\newcommand{\af}{\mathrm{af}}
\newcommand{\al}{\alpha}

\newcommand{\C}{{\mathbb C}}
\newcommand{\cO}{{\mathcal O}}

\newcommand{\Hom}{\mbox{\rm Hom}}

\newcommand{\Ext}{\mbox{\rm Ext}}

\newcommand{\tI}{\mathtt I}

\newcommand{\id}{\mbox{\rm id}}
\newcommand{\la}{\lambda}

\newcommand{\Sym}{\mathfrak S}

\newcommand{\g}{\mathfrak{g}}
\newcommand{\gb}{\mathfrak{b}}

\newcommand{\h}{\mathfrak{h}}
\newcommand{\wh}{\widehat{\mathfrak{h}}}
\newcommand{\wth}{\widetilde{\mathfrak{h}}}
\newcommand{\ws}{\widehat{\mathfrak{s}}}
\newcommand{\wts}{\widetilde{\mathfrak{s}}}

\newcommand{\gI}{\mathfrak{I}}
\newcommand{\gn}{\mathfrak{n}}

\newcommand{\wg}{\widehat{\mathfrak{g}}}
\newcommand{\tg}{\widetilde{\mathfrak{g}}}

\newcommand{\bv}{\mathbf{v}}
\newcommand{\bu}{\mathbf{u}}

\newcommand{\Q}{\mathbb{Q}}

\newcommand{\Z}{\mathbb{Z}}

\newcommand{\vac}{\ket{0}}

\title{A Weyl module stratification of integrable representations\footnote{With an appendix by Ryosuke Kodera}}
\author{Syu \textsc{Kato}\footnote{Department of Mathematics, Kyoto University, {\tt{E-mail:syuchan@math.kyoto-u.ac.jp}}; Research supported in part by JSPS Grant-in-Aid for Scientific Research (B) JP26287004} \and
Sergey Loktev \footnote{Department of Mathematics, National Research University Higher School of Economics, Usacheva str. 6,  Moscow 119048, Russia;  ITEP, 25 B.Cheremushkinskaya, Moscow 117218 Russia; {\tt{E-mail:s.loktev@gmail.com}}; Research supported in part by Laboratory of Mirror Symmetry NRU HSE, RF Government grant, ag. ¹ 14.641.31.0001}}

\begin{document}
\maketitle

\begin{abstract}
We construct a  filtration  on  an integrable highest weight module of an affine Lie algebra whose adjoint graded quotient
is a direct sum of global Weyl modules. We show that the graded multiplicity of each Weyl module there is given by
the corresponding level-restricted Kostka polynomial.  This leads to an interpretation of level-restricted Kostka polynomials as the graded dimension of the space of conformal coinvariants. In addition, as an application of the level one case of the main result, we realize global Weyl modules of current algebras of type $\mathsf{ADE}$ in terms of Schubert subvarieties of thick affine Grassmanian, as predicted by Boris Feigin.
\end{abstract}

\section*{Introduction}
Let $\g$ be a simple Lie algebra over $\C$. Associated to this, we have an untwisted affine Kac-Moody algebra $\widetilde{\g}$ and a current algebra $\g [z] := \g \otimes \C [z] \subset \widetilde{\g}$. The representation theory of $\widetilde{\g}$ attracts a lot of attention in 1990s because of its relation to mathematical physics \cite{FFR94,Tel95,KMS2,KMS95} in addition to its own interest \cite{Kac, KT95}. There, they derive numerous interesting equalities and combinatorics. The representation theory of $\g [z]$ and its variants are studied in detail by many people from 2000s \cite{Kas02,CP01,CG07,FL07,CI15,LNSSS2} as a fork project on the study of representation theory of $\widetilde{\g}$.

The representation theory of $\g [z]$ is essentially the same as $\g [z,z^{-1}]$, that can be also seen as a part of the representation theory of $\widetilde{\g}$, and it also affords the natural space of intertwiners of representations of $\widetilde{\g}$. Moreover, $\g [z]$ inherits many representations directly from $\g$. Therefore, the representation theory of current algebra can be seen as a bridge between these of $\g$ and $\widetilde{\g}$. The goal of this paper is to provide several basic results that support this idea, and to connect them with the combinatorics/equalities from 1990s.

We have an integral weight lattice $P$ of $\g$ and the integral weight lattice $\widetilde{P}$ of $\widetilde{\g}$, such that we have a canonical surjection $\widetilde{P} \in \Lambda \mapsto \overline{\Lambda} \in P$. 
%Let $\rho^\vee$ be the half sum of positive coroots of $\g$.
 Let $P_+$ be the set of dominant weights, and let $P_+^k \subset P$ be the set of level $k$ dominant weights (projected to $P$). For each $\lambda \in P_+$, we have a local Weyl module $W ( \lambda, 0)$ and a global Weyl module $W ( \lambda )$ defined by Chari-Pressley \cite{CP01}.

Both local Weyl modules and global Weyl modules constitute a basis in $K_0$ of the category of graded $\g [z]$-modules. It is shown (\cite{Kat15,Kho13,CI15,Kle14,LNSSS3}) that these bases are orthogonal to each other with respect to ${\rm Ext}^\bullet$ scalar product (see
 Theorem~\ref{crit-filt}) and their characters are related to the  specializations of Macdonald polynomials at $t=0$. It appears
 that characters of some natural representations of $\g [z]$ can be expressed via Weyl module characters with positive coefficients.
 In particular, the Cauchy identity implies that the projective $\g [z]$-modules have a filtrations whose adjoint graded space consists of  global Weyl modules (\cite{CI15}).

Let us choose an inclusion of $\g [z]$ into $\widetilde{\g}$ as the {\it nonpositive} part of $\widetilde{\g}$, so the representations of $\widetilde{\g}$ can be considered as $\g [z]$-modules. For each $\lambda \in P_+^k$, we have a level $k$ integrable highest weight $\widetilde{\g}$-module $L_k ( \lambda )$. In \cite{CF13}, it is shown that the characters of $L_1 ( \Lambda )^{\otimes n}$ ($n \ge 1$) can be expressed via characters of the global Weyl modules with positive coefficients.

Our main result provides an explanation of this phenomena: 

\begin{fthm}\label{fstr}
Each integrable highest weight $\widetilde{\g}$-module considered as a $\g [z]$-module
admits a filtration by global Weyl modules.
\end{fthm}

Let $M$ be a finitely generated graded $\g [z]$-module stratified by global Weyl modules. We denote the graded multiplicity of $W ( \mu )$ in $M$ by $( M : W ( \mu ) )_q$. Theorem \ref{fstr} implies that we have a well-defined notion of the graded multiplicity $( L_{k} ( \lambda ) : W ( \mu ) )_q$ that counts the number of occurrences of $W ( \mu )$ in the filtration of $L_{k} ( \lambda )$ (with grading shifts counted).

Since the representation theory of $\g [z]$ naturally carries the information of Macdonald polynomials specialized to $t = 0$, Theorem \ref{fstr} and a version of the BGG reciprocity (Theorem \ref{rec} and Theorem \ref{crit-filt}) yields:

\begin{fthm}\label{fPhom}
For each $\lambda \in P_+^k$ and $\mu \in P_+$, we have
$$\mathsf{gdim} \, H_i ( \g [z], \g ; W ( - w_0 \mu, 0 ) \otimes L_{k} ( \la ) ) = \begin{cases} ( L_{k} ( \la ) : W ( \mu ) )_q & (i = 0)\\ 0 & (i \neq 0)\end{cases},$$
where $w_0$ denote the longest element in the Weyl group of $\g$. In addition, $( L_{k} ( \la ) : W ( \mu ) )_q$ coincides with the restricted Kostka polynomial of level $k$ defined combinatorially in \S \ref{comb-mac}.
\end{fthm}

Theorem \ref{fPhom} can be seen as a direct extension of Feigin-Feigin \cite{FF05}. The same idea as in the proof of Theorem \ref{fPhom} also yield a proof of Teleman's Borel-Weil-Bott theorem \cite[Theorem 0 a)]{Tel95} that we include for the sake of reference ($=$ Corollary \ref{TBWB}). From this view point, it is natural to reformulate our results in terms of conformal coinvariants.

\begin{fthm}\label{fCFT}
For each $\lambda \in P_+^k$ and $\mu \in P_+$, the vector space
$$H_0 ( \g [z], \g ; W ( \mu ) \otimes L_{k} ( \lambda ) )^{\vee}$$
is a free module over $\C [\A^{(\mu)}]$, where $\A^{(\mu)}$ is a certain configuration space of $|\la|$-points in $\A^1$ $($see \S $\ref{rep-current})$. Its specialization to each point $\vec{x} \in \A^{(\mu)}$ gives the space of the generalized conformal coinvariants $($see \S $\ref{lrm})$.
\end{fthm}

We remark that if we specialize to a generic $\vec{x}$, then Theorem \ref{fCFT} reduces into Teleman's result \cite{Tel95}. In general, the above homology group has subtle cancelations that is observed in \cite{FF05} when $\g = \mathfrak{sl} ( 2 )$.

Assume that $\g$ is of type $\mathsf{ADE}$. Then, the $G$-invariant Schubert subvariety of the thick affine Grassmanian $\mathbf{Gr} _G$ of $G$ is in bijection with $P_+$ (see e.g. \cite{Kas90, Zhu09}). We denote by $\mathbf{Gr} _G ^{\lambda}$ the Schubert subvariety of $\mathbf{Gr} _G$ corresponding to $\la \in P_+$. Our analysis, together with that of Cherednik-Feigin \cite{CF13} and a result of \cite{Kat17b}, implies the following realization of global Weyl modules predicted by Boris Feigin:

\begin{fthm}\label{fFeigin}
Assume that $\g$ is of type $\mathsf{ADE}$. For each $\Omega \in P_+^1$ and $\lambda \in P_+$ such that $\lambda \ge \overline{\Omega}$, we have the following isomorphism of $\g [z]$-modules:
$$\Gamma_c ( \mathbf{Gr} _G ^{\lambda}, \cO_{\mathbf{Gr} _G ^{\lambda}} ( 1 ) ) \stackrel{\cong}{\longrightarrow} W ( \lambda )^\vee,$$
where $\cO_{\mathbf{Gr} _G ^{\lambda}} ( 1 )$ is the determinant line bundle of $\mathbf{Gr}_G$ restricted to $\mathbf{Gr} _G ^{\lambda}$.
\end{fthm}

The organization of the paper is as follows: In section one, we prepare basic notation and environments. In section two, we exhibit that the Bernstein-Gelfand-Gelfand-Lepowsky resolution gives a projective resolution of an integrable highest weight module, viewed as a $\g [z]$-module. In section three, we define the level-restricted Kostka polynomials and provide its main properties (Theorem \ref{fstr} and Theorem \ref{fPhom}). In addition, we identify our coinvariant space as a natural enhancement of conformal coinvariants (Theorem \ref{fCFT}). In section four, we derive Feigin's realization of global Weyl modules in terms of a Schubert subvariety of thick affine Grassmanian. In section five, we provide an alternating sum formula (Theorem \ref{P-char}) for the polynomials that naturally extends the level-restricted Kostka polynomials (that is implicit in the literature), which plays a crucial role in the comparison with our level-restricted Kostka polynomials (Corollary \ref{comp-mac}). The appendix (by Ryosuke Kodera) contains a bypass of 
Theorem \ref{fstr} in the proof of Theorem \ref{fFeigin} using free field realizations.

\medskip

{\small{\bf Acknowledgments:}
S.K. thanks Michael Finkelberg to communicate him Feigin's insight and a kind invitation to Moscow on the fall 2016, and Masato Okado to show/explain his mimeo. S.L. thanks Department of Mathematics of Kyoto University for their hospitality.  
We both thanks Boris Feigin and Ivan Cherednik for fruitful and stimulating discussions.}

\section{Preliminaries}\label{prelim}
A vector space is always a $\C$-vector space, and a graded vector space refers to a $\Z$-graded vector space whose grading is either bounded from the below or bounded from the above and each of its graded piece is finite-dimensional. For a graded vector space $M = \bigoplus_{i \in \Z} M _i$ or its completion $M^{\wedge} = \prod_{i \in \Z} M_i$, we define its dual as $M^\vee := \bigoplus_{i \in \Z} \mathrm{Hom}_{\C} ( M _i, \C )$, where $\mathrm{Hom}_{\C} ( M _i, \C )$ is understood to have degree $-i$. We define the graded dimension of a graded vector space as
$$\mathsf{gdim} \, M := \sum_{i\in \Z} q^i \dim _{\C} \, M_i \in \Q [\![q,q^{-1}]\!].$$
For each $n \in \Z$, let us define the grade $n$-shift of a graded vector space $M$ as: $( M \left< n \right> )_i := M_{n+i}$ for every $i \in \Z$. For $f ( q ) \in \Q (q)$, we set $\overline{f ( q ) } := f ( q^{-1} )$.

\subsection{Algebraic groups and its Lie algebras}
Let $G$ be a connected, adjoint semi-simple algebraic group over $\C$, and let $B$ and $H$ be a Borel subgroup and a maximal torus of $G$ such that $H \subset B$. Let $G_{\rm{sc}}$ be the simply connected cover of $G$, and $H_{\rm{sc}}$ be the preimage of $H$ in $G_{\rm{sc}}$. We set $N$ $(= [B,B])$ to be the unipotent radical of $B$ and let $N^-$ be the opposite unipotent subgroup of $N$ with respect to $H$. We denote the Lie algebra of an algebraic group by the corresponding German letter. We have a (finite) Weyl group $W := N_G ( H ) / H$. For an algebraic group $E$, we denote its set of $\C [z]$-valued points by $E [z]$, its set of $\C [\![z]\!]$-valued points by $E [\![z]\!]$, and its set of $\C (z)$-valued points by $E (z)$.

Let $P := \mathrm{Hom} _{gr} ( H_{\rm{sc}}, \C ^{\times} )$ be the weight lattice of $H_{\rm{sc}}$, let $\Delta \subset P$ be the set of roots, let $\Delta^+ \subset \Delta$ be the set of roots belonging to $\gb$, and let $\Pi \subset \Delta^+$ be the set of simple roots. We denote by $\Pi^{\vee}$ the set of simple coroots of $\g$. For $\lambda, \mu \in P$, we define $\lambda \ge \mu$ if and only if $\lambda - \mu \in \Z_{\ge 0} \Delta^+$. Let $Q^{\vee}$ be the dual lattice of $P$ with a natural pairing $\left< \bullet, \bullet \right> : Q^{\vee} \times P \rightarrow \Z$. Let $r$ be the rank of $G$ and we set $\mathtt I := \{1,2,\ldots,r\}$. We fix bijections $\mathtt I \cong \Pi \cong \Pi^{\vee}$ such that $i \in \mathtt I$ corresponds to $\alpha_i \in \Pi$, its coroot $\alpha_i^{\vee} \in \Pi^{\vee}$, (non-zero) root vectors $E_i, F_i$ corresponding to $\alpha_i, - \alpha_i$, and a simple reflection $s_i \in W$ corresponding to $\alpha_i$. We also define a reflection $s_{\alpha} \in W$ corresponding to $\alpha \in \Delta^+$. Let $\ell : W \rightarrow \Z_{\ge 0}$ be the length function and let $w_0 \in W$ be the longest element. Let $P_{+} := \{ \la \in P \mid \left< \alpha_{i}^{\vee}, \la \right> \in \Z_{\ge 0}, i \in \mathtt I \}$. Let $\{\varpi_i\}_{i \in \tI} \subset P_+$ denote the dual basis of $\Pi^{\vee}$. For each $\la \in P_{+}$, we have a finite-dimensional irreducible $\g$-module $V ( \la )$ with highest weight $\la$. 

Let $\Delta_{\af} := \Delta \times \Z \delta \cup \Z_{\neq 0} \delta$ be the untwisted affine root system of $\Delta$ with its positive part $\Delta^+ \subset \Delta_{\af}^{+}$. We set $\alpha_0 := - \vartheta + \delta$, $\Pi_{\af} := \Pi \cup \{ \alpha_0 \}$, and $\tI_{\af} := \tI \cup \{ 0 \}$, where $\vartheta$ is the highest root of $\Delta^+$. We define a normalized inner product $( \bullet, \bullet ) : \h^* \times \h^* \rightarrow \C$ to be the unique $W$-invariant inner product such that $( \vartheta, \vartheta ) = 2$. We set $W _{\af} := W \ltimes Q^{\vee}$ and call it the affine Weyl group. It is a reflection group generated by $\{s_i \mid i \in \mathtt I_{\af} \}$, where $s_0$ is the reflection with respect to $\alpha_0$. This equips $W_\af$ a length function $\ell : W_\af \rightarrow \Z_{\ge 0}$ extending that of $W$. For a subgroup $W' \subset W_\af$ generated by a subset of $\{s_{i}\}_{i \in \tI}$, we identify $W' \setminus W_\af$ with the set of minimal length representatives of right $W'$-cosets in $W_\af$. The embedding $Q^{\vee} \hookrightarrow W_\af$ defines a translation element $t_{\gamma}$ for each $\gamma \in Q^{\vee}$ whose normalization is $t_{- \vartheta^{\vee}} = s_{\vartheta} s_0$.

\subsection{Affine Lie algebras}\label{aff-Lie}

Let $\widetilde{\g}$ be the untwisted affine Kac-Moody algebra associated to $\g$. I.e. we have
$$\widetilde{\g} = \g \otimes _\C \C [\xi, \xi^{-1}] \oplus \C K \oplus \C d,$$
where $K$ is central, $[d, X \otimes \xi^m] = m X \otimes \xi ^m$ for each $X \in \g$ and $m \in \Z$, and for each $X, Y \in \g$ and $f, g \in \C [\xi^{\pm 1}]$ it holds:
$$[X \otimes f , Y \otimes g ] = [X, Y] \otimes f g + ( X, Y )_{\g} \cdot K \cdot \mathrm{Res}_{\xi = 0} f \frac{\partial g}{\partial \xi},$$
where $(\bullet, \bullet)_{\g}$ denotes the unique  $\g$-invariant bilinear form such that $( \al^{\vee}, \al^{\vee})_{\g} = 2$ for a long simple root $\al$. We set $E_0 := F_{\vartheta} \otimes \xi$ and $F_0 := E_{\vartheta} \otimes \xi^{-1}$ (these are root vectors of $\alpha_0, - \alpha_0$, respectively). Then, $\{E_i, F_i\} _{i \in \tI _{\af}}$ generates a subalgebra $\wg$ of $\tg$. We set $\wh := \h \oplus \C K$ and $\wth := \wh \oplus \C d$. For each $i \in \tI_{\af}$, we denote by $\alpha_{i}^{\vee} \in \wth$ the coroot of $\al_{i}$ whose set enhances $\Pi^{\vee} \subset \h$. Let $\gI \subset \tg$ (resp. $\gI^-$) be the subalgebra of $\tg$ generated by $\wth$ and $\{ E_i \}_{i \in \mathtt I_\af}$  (resp.
 $\wth$ and $\{ F_i \}_{i \in \mathtt I_\af}$). 

By introducing $z = \xi^{-1} \in \C [\xi^{\pm 1}]$, we define a Lie subalgebra $\g [z] \subset \tg$. We have the Chevalley involution $\theta$ on $\tg$ such that $\theta ( E_i ) = F_i, \theta ( F_i ) = E_i$ ($i \in \tI_\af$) and $\theta ( d ) = - d$. We also have an involution $\Phi$ on $\tg$ such that $\Phi ( X \otimes \xi ^m ) := X \otimes z ^m$ ($m \in \Z$), $\Phi ( K ) := -K$, and $\Phi ( d ) := - d$. Note that $\Phi$ induces an isomorphism $\g [\xi] \cong \g [z]$ as Lie subalgebras of $\tg$. We set $\Theta := \theta \circ \Phi = \Phi \circ \theta$.

We set $\widehat{P}$ to be the lattice spanned by a fixed choice of fundamental weights $\Lambda_0, \Lambda_1, \ldots, \Lambda_r \in \wth^*$ of $\wg$ such that $\left< \al_i ^{\vee}, \Lambda_j \right> = \delta_{ij}$. We define $\widehat{P}_+ := \sum_{i \in \mathtt I_\af} \Z_{\ge 0} \Lambda_i$ and $\widehat{P} := \sum_{i \in \mathtt I_\af} \Z \Lambda_i$. For $k \in \Z_{\ge 0}$, we also set
$$\widehat{P}_+^k := \{ \Lambda \in \widehat{P}_+ \mid \left< K, \Lambda \right> = k \} \subset \{ \Lambda \in \widehat{P} \mid \left< K, \Lambda \right> = k \} =: \widehat{P}^k.$$
We define $\widetilde{P} := \widehat{P} \oplus \Z \delta \subset \wth^*$. We have projection maps
$$\widetilde{P} \rightarrow \widetilde{P} / \Z \delta \cong \widehat{P}\rightarrow \widehat{P} / \Z \Lambda_0 \cong P.$$
We denote the projection of $\Lambda \in \widetilde{P}$ or $\widehat{P}$ to $P$ by $\overline{\Lambda}$. Let us denote the image of $\widehat{P}_+^k$ under this projection by $P_+^k$. We also identify $P$ with $\widehat{P}^0 \subset \widetilde{P}$.

Let $k \in \Z$. We fix an element $\rho_k \in \widetilde{\h}^*$ such that $\left< \alpha_i^{\vee}, \rho_k \right> = 1$ for each $i \in \mathtt I$, and $\left< \alpha_0^{\vee}, \rho_k \right> = k + 1$. For each $w \in W_{\af}$ and $\Lambda \in \widetilde{P
}$, we define
$$w \circ _k \Lambda := w ( \Lambda + \rho_k ) - \rho_k.$$
When $k = 0$, then we simply write $\circ$ instead of $\circ_k$. Note that for $w \in W_{\af}$ and $\lambda \in P$, we have $w \circ ( \lambda + k \Lambda_0 ) = w \circ_{k} \lambda + k \Lambda_{0}$.

Every element of $P$ is either $\circ_k$-conjugate to $P_+^k$ modulo $\Z \delta$ or $W_\af$ has a non-trivial stabilizer group with respect to the $\circ_k$-action.

Finally, we set $U_k ( \wg ) := U ( \wg ) / ( K - k ) U ( \wg )$ and $U_k ( \tg ) := U ( \tg ) / ( K - k ) U ( \widetilde{\g} )$. We refer $U_k ( \wg )$-modules and $U_k ( \tg )$-modules by the $\wg_k$-modules and $\tg_k$-modules, respectively.

\subsection{Representations of current algebras}\label{rep-current}

We review some results from current algebra representations (cf. \cite{Kat16} \S 1.2).

\begin{defn}[$\g$-integrable module]
A $\g [z]$-module $M$ is said to be $\g$-integrable if $M$ is finitely generated and it decomposes into a sum of finite-dimensional $\g$-modules. Let $\g [z] \mathchar`-\mathsf{mod}$ be the category of $\g$-integrable $\g [z]$-modules. For each $\lambda \in P_+$, let $\g [z] \mathchar`-\mathsf{mod}^{\le \lambda}$ be the full subcategory of $\g [z] \mathchar`-\mathsf{mod}$ whose object is isomorphic to a direct sum of $\g$-modules in $\{ V (\mu) \}_{\mu \le \lambda}$.\\
A $\g [z]$-module $M$ is said to be graded if $M$ is a graded vector space and we have $( X \otimes z^{n} ) M_{m} \subset M_{n+m}$ for each $X \in \g$ and $n, m \in \Z$. We denote the category of graded $\g$-integrable $\g [z]$-modules by $\g [z] \mathchar`-\mathsf{gmod}$.
\end{defn}

\begin{defn}[projective modules and global Weyl module]\label{gWP}
For each $\lambda \in P_+$, we define the non-restricted projective module $P ( \lambda )$ as
$$P ( \lambda ) := U ( \g [z] ) \otimes _{U ( \g )} V ( \lambda ).$$
Let $P ( \lambda ;\mu )$ be the largest $\g [z]$-module quotient of $P ( \lambda )$ such that
\begin{equation}
  \Hom _{\g} ( V (\gamma), P ( \lambda; \mu ) ) = \{ 0 \} \hskip 5mm \text{if} \hskip 5mm \gamma \not\le \mu.\label{cut}
\end{equation}
We define the global Weyl module $W ( \lambda )$ of $\g [z]$ to be $P ( \lambda; \lambda )$.
\end{defn}

For each $\la \in P_{+}$ and $a \in \C$, the $\g$-module $V ( \la )$ can be regarded as a $\g [z]$-module through the evaluation map $\g [z] \to \g$ at $z = a$ for each $\la \in P_{+}$. We denote it by $V ( \la, a )$. We can further regard $V ( \la, 0 )$ as a module in $\g [z] \mathchar`-\mathsf{gmod}$ by putting a grading concentrated in a single degree. For $M \in \g [z] \mathchar`-\mathsf{gmod}$, we have $M \left< n \right> \in \g [z] \mathchar`-\mathsf{gmod}$ for every $n \in \Z$.

\begin{lem}\label{graded-mods}
Let $\lambda, \mu \in P_+$. The projective module $P ( \lambda )$, its quotient $P ( \lambda; \mu )$, and global Weyl modules $W ( \lambda )$ can be regarded as graded modules with simple heads $V ( \lambda, 0 )$ sitting at degree $0$. \hfill $\Box$
\end{lem}

\begin{lem}
The module $P ( \lambda )$ is the projective cover of $V ( \lambda, x )$ as a $\g$-integrable $\g [z]$-module for every $x \in \C$.
\end{lem}

\begin{proof}
Thanks to the $\g$-integrability, $P ( \la )$ is the maximal $\g$-integrable quotient of $U ( \g [z] )$ that is generated by $V ( \la )$ (as $U ( \g )$-modules).
\end{proof}

In order to establish relation between global and local Weyl modules, we set $|\la| := \sum_{i \in \tI} \left< \al_i^{\vee}, \la \right>$ and
$$\C [\A^{(\la)}] := \bigotimes _{i \in \tI} \C [X^{(i)}_1,X^{(i)}_2,\ldots,X^{(i)}_{m_i}]^{\Sym_{m_i}} \hskip 5mm m_i = \left< \al_i^{\vee}, \la \right>,$$
where we understand that $\C [\A^{(\la)}]$ is a graded ring by setting $\deg \, X^{(i)}_j = 1$ for every $i \in \tI$ and $j \in \{ 1,2,\ldots,m_i\}$.

\begin{thm}[Chari-Fourier-Khandai \cite{CFK}]
The module $W ( \lambda )$ admits a free action of $\C [ \A ^{(\lambda)} ]$ induced by the $U ( \h [z] )$-action on the $\h$-weight $\lambda$-part of $W ( \lambda )$, that commutes with the $\g [z]$-action and respects the grading of $W ( \lambda )$.
\end{thm}

\begin{defn}\label{loc-w}
For each $x \in \A^{(\lambda)}$, we have a specialization $W ( \lambda, x ) := W ( \lambda ) \otimes_{\C [\A^{(\lambda)}]}  \C _x$. These modules are called {\em local Weyl modules}.
\end{defn}

\begin{lem}[see \cite{CFK}]
If $x \in \A^{(\lambda)}$ is the orbit of $|\lambda|$-distinct points, then we have
$$W ( \lambda, x ) \cong \bigotimes _{i=1}^r \bigotimes _{j=1}^{m_i} W ( \varpi_i, x_{i,j} ).$$
Here $\{ ( x_{i,1},\ldots,x_{i,m_i} )\}_{i \in \tI} \in \prod_{i \in \tI}\A ^{m_i}$ corresponds to $x$ $($up to $\prod_i \Sym_{m_i}$-action$)$.
\end{lem}

\begin{thm}[Chari-Loktev \cite{CL06}, Fourier-Littelmann \cite{FL07}, Naoi \cite{Nao12}]\label{free}
The action of $\C [ \A ^{(\lambda)} ]$ on $W ( \lambda )$  is free, so $W ( \lambda, x ) \cong W ( \lambda, y )$ as $\g$-modules $($but not as $\g[z]$-modules$)$ for each $x, y \in \A^{(\lambda)}$. In particular, $W ( \lambda, x )$ is finite-dimensional for any $x$.
\end{thm}

\subsection{Verma and parabolic Verma modules}\label{VW}

We set $\g [z]_1 := z \g[z] = \ker ( \g [z] \to \g ) \cong \g \otimes z\C[z]$ and $\g [\xi]_1 :=  \xi \g[\xi] = \Phi ( \g [z]_1 )$.

\begin{defn}\label{po}
We say that a $\tg$-module $M$ belongs to the parabolic category $\cO$ if 
\begin{enumerate}
\item The action of $\wth$ on $M$ is semi-simple. Moreover, each $\wth$-weight space is finite-dimensional and eigenvalues of $d \in \wth$ are bounded from above;
\item The action of  $\g [\xi]$ is locally finite (each vector generates a finite-dimesional space under this action).
\end{enumerate}
\end{defn}

\begin{rem}
{\bf 1)} The restriction a module in the parabolic category $\cO$ to $\g [z]$ is $\g$-integrable. {\bf 2)} One can replace the second condition in Definition~\ref{po} with locally finiteness of the action of $\gI$ to obtain the usual category $\cO$.
\end{rem}

\begin{defn}[integrable modules]
A module from the parabolic category $\cO$ is said to be integrable if the actions of $E_0$ and $F_0$ are locally finite.
\end{defn}

Let $\lambda \in P_{+}$ and $k \in \Z$. We have the corresponding one-dimensional module $\C_{( \lambda + k \Lambda_{0} )}$ of $\wth$, that can be inflated to a module of $\gI$ by the trivial action of $[\gI, \gI]$. We define the Verma module $\mathbf M_k ( \lambda )$ as
$$\mathbf M_k ( \lambda ) := U ( \tg ) \otimes_{U ( \gI )} \C_{( \lambda + k \Lambda_{0} )},$$
and its (unique) simple quotient by $L _{k} ( \lambda )$ (see e.g. \cite[\S 1.2]{KT95}). In case $\lambda \in P_+$, we have a finite-dimensional irreducible $\g + \wth$-module $V ( \lambda ) \otimes _\C \C_{k \Lambda_0}$ with a $\gb$-eigenvector $\bv_{\lambda}$ of $\wth$-weight $\lambda + k \Lambda_0$. We inflate $V ( \lambda )$ into a $( \g [\xi] + \wth )$-module by the trivial action of $\g [\xi]_1$. We define the parabolic Verma module $M_{k} ( \lambda )$ as
$$M_k ( \lambda ) := U ( \widetilde{\g} ) \otimes_{U ( \g [\xi] + \wth )} \left( V ( \lambda ) \otimes _\C \C_{k \Lambda_0} \right).$$
Twisting $\C_{\pm \delta}$ on the RHS gives rise to isomorphic $\wg_k$-modules (up to the shift of the $\C d$-action). We understand this twist as the grading shift of $\wg_k$-modules.

By construction, we have $\widetilde{\g}$-module surjections
$$\mathbf M_{k} ( \lambda ) \longrightarrow M_{k} ( \lambda ) \longrightarrow L_{k} ( \lambda )$$
for each $\lambda \in P_{+}$. We understand that their cyclic vectors are degree $0$ by convention. The following facts are crucial in what follows:

\begin{thm}[see e.g. \cite{Pol91, KL93}]
For each $\la \in P_+$, the modules $M_k ( \lambda ) $ and $L_k  ( \lambda )$ belong to the parabolic category $\cO$. Moreover, $L_k  ( \lambda )$ is integrable if and only if $\lambda \in P_+^k$. \hfill $\Box$
\end{thm}

For each $\la \in P_+^k$, the module $L_k ( \la )$ carries a unique $\wth$-eigenvector of its weight $w ( \la + k \Lambda_0 ) \in \wth^*$ for each $w \in W_\af$. We call such a weight an extremal weight and such a vector an extremal weight vector of $L _k ( \la )$ (see e.g. \cite{Kum87}).

\begin{thm}[Chari-Greenstein \cite{CG07}]\label{CG}
For each $\lambda \in P_+$, we have
$$M_{k} ( \lambda ) \cong U ( \g [z]_1 ) \otimes _{\C} V ( \lambda )$$
as graded $\g$-modules. Moreover, $M_{k} ( \lambda )$ is isomorphic to a projective module $P ( \lambda )$ in $\g [z] \mathchar`-\mathsf{mod}$. \hfill $\Box$
\end{thm}

\begin{lem}\label{g-int}
For each $\lambda \in P_+$ and a $\wg$-module $M$ from the parabolic category $\cO$, we have an isomorphism
$$\Hom_{\wg} ( M_{k} ( \lambda ), M ) \cong \Hom_{\wh + \gn} ( \C_{\lambda}, M^{\g [\xi]_1} ).$$
\end{lem}

\begin{proof}
The RHS corresponds to the space $\wg$-module morphisms $\psi : \mathbf M_k ( \lambda ) \rightarrow M$. Consider the $\wg$-module maps
$$\mathbf M_k ( w \circ \lambda ) \longrightarrow \mathbf M_k ( \lambda ) \hskip 5mm  w \in W$$
induced by the $U ( \g )$-module inclusion of their $\g [\xi] _1$-fixed parts. By the classification of finite-dimensional highest weight $\g$-modules, we deduce that $\mathbf M_k ( w \circ \lambda ) ^{\g [\xi]_1}$ has a $\g$-integrable quotient if and only if $w = e$. Since $M$ belongs to the parabolic category $\cO$ and, hence, $\g$-integrable, $\psi$ must factor through the quotient of $\mathbf M_k ( \lambda )$ by the images of $\mathbf M_k ( w \circ \lambda )$ ($w \neq e$), that is $M_k ( \lambda )$. So the assertion follows.
\end{proof}

\begin{thm}[see \cite{Kac}, Chapter 7]\label{verma}
Let $\lambda \in P_+^k$. Then, we have
$$L_k ( \lambda ) \cong M_k ( \lambda ) / U ( \wg ) F_0 ^{k - \left< \vartheta, \lambda \right> + 1} v_{\lambda}.$$
\end{thm}

This isomorphism can be extended to the resolution of $L_k ( \lambda )$ by parabolic Verma modules.

\begin{thm}[Hackenberger-Kolb \cite{HK07}, Theorem 3.6]\label{BGGL-res}
For each $\lambda \in P_+^k$, we have a resolution by parabolic Verma modules
\begin{equation}
\cdots \stackrel{d_3}{\longrightarrow} \bigoplus_{w \in W \backslash W_{\af}, \ell ( w ) = 2} M _k ( w \circ_k \lambda ) \stackrel{d_2}{\longrightarrow} M _k ( s_0 \circ_k \lambda ) \stackrel{d_1}{\longrightarrow} M _k ( \lambda ) \rightarrow L _k ( \lambda ) \rightarrow 0.\label{W-res}
\end{equation}
\end{thm}

Combining this statement with Theorem \ref{CG}, we obtain

\begin{prop}
By restricting to $\g [z]$, the resolution $(\ref{W-res})$ can be seen as a graded projective resolution of $L _k ( \lambda )$ whose grading arises from the $(-d)$-action.
\end{prop}

\section{The Weyl filtration}\label{wf}

We retain the setting of the previous section. In this section, we utilize the following two results on current algebra representations to prove Theorem \ref{FC-str}.

For $M \in \g [z] \mathchar`-\mathsf{gmod}$ and $\la \in P_{+}$, we set
\begin{align*}
[M: V ( \la, 0 )]_{q} &:= \sum_{n \in \Z} q^{n} \cdot \dim \, \Hom_{\g} ( V ( \la ), M _{n} )\\
\mathsf{ch} \, M & := \sum_{n \in \Z, \mu \in P} q^n e^{\mu} \cdot \dim \, \Hom_{\h} ( \C _\mu, M _{n} ).
\end{align*}
In case $M$ is a $\widetilde{\g}_k$-module, then $q$ represents $e^{- \delta}$ in the standard convention.

\begin{thm}[Chari-Ion \cite{CI15}]\label{rec}
\ 

\begin{enumerate}
\item The projective module $P ( \lambda)$ of $\g [z] \mathchar`-\mathsf{mod}$ admits a graded filtration by $\{ W (\mu) \}_{\mu \ge \lambda}$ with suitable grading shifts;

\item If we denote by $( P ( \lambda ) : W ( \mu ) )_q$ the number of $W(\mu)$ appearing in the filtration $($in a graded sense$)$, then for each $\lambda, \mu \in P_+$ we have
$$( P ( \lambda ) : W ( \mu ) )_q = [ W ( \mu,0 ) : V ( \lambda ) ]_q.$$
In particular, the both sides belongs to $\Z_{\ge 0} [q]$.
\end{enumerate}
\end{thm}

\begin{thm}[Kleshchev \cite{Kle14} Theorem 7.21 and Lemma 7.23, cf. \S 10.3]\label{crit-filt}
\ 

\begin{enumerate}
\item We have
$$ \mathrm{Ext}^i_{\g [z] \mathchar`-\mathsf{mod}} ( W ( \lambda ), W (\mu,0)^* \left< n \right> ) = \begin{cases} \C & (\lambda = - w_0 \mu, i = 0, n = 0)\\ \{ 0 \} & (\text{otherwise}) \end{cases};$$
\item A finitely generated $\g$-integrable $\g [z]$-module $M$ admits a filtration by global Weyl modules if and only if
$$\mathrm{Ext} ^1_{\g [z]} ( M, W ( \lambda, 0 ) ^* ) = \{ 0 \} \hskip 5mm \forall \lambda \in P _+;$$
\item A finite-dimensional graded $\g [z]$-module $M$ admits a filtration by the dual of local Weyl modules if and only if
$$\mathrm{Ext} ^1_{\g [z]} ( W ( \lambda ), M ) = \{ 0 \} \hskip 5mm \forall \lambda \in P _+.$$
\end{enumerate}
\end{thm}

\begin{rem}\label{ext-id}
For each $M,M' \in \g [z] \mathchar`-\mathsf{mod}$, we have
$$\mathrm{Ext} ^i _{\g [z] \mathchar`-\mathsf{mod}} ( M, M' ) \stackrel{\cong}{\longrightarrow} \mathrm{Ext} ^i _{\g [z]} ( M, M' ) \hskip 5mm i \in \Z$$
since the category of finite-dimensional $\g$-modules is completely reducible (cf. \cite[\S 3.1]{Kum02} or \cite[Remark 0.2]{FGT}; see also Remark \ref{hom-rem}).
\end{rem}

\subsection{The Demazure property}

Note that we can consider the current algebra representation as the modules of $( \g [z] + \wth )$ with $K=0$ and the action of $(-d)$ as the grading operator.

\begin{defn}\label{DP}
Let $V$ be a representation of a subalgebra of $\tg$ contating $\wth$ and $E_0$. We say that $V$ satisfies the Demazure property if for each $n>0$ and $\gamma \in P$ such
that $\left< \alpha_0^{\vee}, \gamma \right> = -n$, the action of $\left(E_0\right)^{n-1}$ on 
the corresponding weight space $V^\gamma = \C_{\gamma} \otimes \mathrm{Hom}_{\h} ( \C_{\gamma}, V ) \subset V$
has trivial kernel.
\end{defn}

\begin{rem}
{\bf 1)} This definition is motivated by the fact that for the simply-laced $\g$, local Weyl modules coincide with the level one Demazure modules, and this property can be deduced from the structure of their Demazure crystals. {\bf 2)} The assumption on the Lie algebra of Definition \ref{DP} holds for $\Phi ( \g [z] + \wth ) = ( \g [\xi] + \wth )$, and $\gI$. 
\end{rem}

Let us introduce some notation. For each $i \in \mathtt I_{\af}$, we have a minimal paraholic subalgebra $\gI \subset \gI ( i ) = \C F_i \oplus \gI$. We denote by $\mathfrak{sl} ( 2, i )$ the Lie subalgebra of $\tg$ generated by $E_i$ and $F_i$, and $\gb ( i )$ the Lie subalgebra of $\tg$ generated by $E_i$ and $\wth$. A string $\gb ( i ) $-module is a finite-dimensional $\gb ( i )$-module that is $\wth$-semisimple and have unique simple submodule and unique simple quotient (that is automatically isomorphic to an irreducible $\mathfrak{sl} ( 2, i )$-module up to a $\wth$-weight twist).

In the below (only) in this subsection, we regard $W ( \la )$ as a $\g [\xi]$-module through $\Phi$ for the sake of conventional simplicity. (This modification is rather consistent with the other literatures.)

\begin{thm}[\cite{Kat16} Theorem 4.12 2)]\label{FM}
For each $\la \in P_+$, there exists an $\gI ( 0 )$-module $W ( \lambda )_{s_0}$ such that
$$W ( \lambda ) \subset W ( \lambda )_{s_0}.$$
\end{thm}

The $\gI ( 0 )$-module $W ( \lambda )_{s_0}$ is completely reducible as an $\mathfrak{sl} ( 2, 0 )$-module. For each $m \in \Z_{\ge 0}$, let $W ( \lambda )_{s_0}^m$ denote the sum of irreducible $\mathfrak{sl} ( 2, 0 )$-submodules of dimension $m$. We have $W ( \la )_{s_0} = \bigoplus_{m \ge 0} W ( \la )_{s_0}^m$.

We define the canonical filtration of $W ( \lambda )_{s_0}$ as:
$$F_n W ( \lambda )_{s_0} := \sum_{m \ge n} W ( \lambda )_{s_0}^m.$$
This defines a decreasing separable filtration of $W ( \lambda )_{s_0}$ such that $F_0 W ( \lambda )_{s_0} = W ( \lambda )_{s_0}$.

\begin{thm}\label{crystal}
For each $\la \in P_+$ and $m \ge 0$, the module
$$( F_m W ( \la )_{s_0} \cap W ( \la ) ) / ( F_{m+1} W ( \la )_{s_0} \cap W ( \la ) )$$
constructed from the inclusion in Theorem $\ref{FM}$ is $\mathfrak{b} ( 0 )$-stable, and it is the direct sum of irreducible $( \mathfrak{sl} ( 2, 0 ) + \wth )$-modules of dimension $m$, and one-dimensional $\wth$-modules $\C_{\mu}$ such that $\left< \al_0^{\vee}, \mu \right> = m-1$.
\end{thm}

\begin{proof}
The assertion follows by \cite[Lemma 4.4 and Corollary 4.8]{Kat16}.
\end{proof}

\begin{thm}\label{w-free}
Keep the setting of Theorem \ref{crystal} and Theorem \ref{free}. Then, the $\C [\A^{(\lambda)}]$-action on $W ( \la )$ naturally induces $\mathfrak{b} ( 0 )$-endomorphisms of $W ( \lambda )_{s_0}$. Moreover, this $\C [\A^{(\lambda)}]$-action on $W ( \la )_{s_0}$ is free.
\end{thm}

\begin{proof}
The first part of the assertion follows by Theorem~\ref{free} since $\C [\A^{(\lambda)}]$ is also isomorphic to the $\h$-weight $s_{\vartheta} \la$-part of $W ( \lambda )$ (by the $\g$-invariance) and $W ( \lambda )_{s_0}$ is cyclically generated by the $\h$-weight $s_{\vartheta} \la$-part inside $W ( \lambda )$ by \cite[Proof of Theorem 5.1]{Kat16}. The freeness assertion is \cite[Theorem 5.1]{Kat16}.
\end{proof}

\begin{lem}\label{glWD}
Global and local Weyl modules satisfy the Demazure property if we twist the action by $\Phi$ $($as above$)$.
\end{lem}

\begin{proof}
We first check the Demazure property for the global Weyl module $W ( \lambda )$. Suppose $\bv \in  W ( \lambda )$
is a non-zero vector of ($\wth$-)weight $\gamma$ with $\left< \alpha_0^{\vee}, \gamma \right> = -n < 0$. Then 
$\bv \in F_m W ( \la )_{s_0} \cap W ( \la )$ with some $m > n$. As $\left< \alpha_0^{\vee}, \gamma \right> < 0$, the only possibility provided by Theorem~\ref{crystal} is that $\bv$ belongs to $F_m W ( \la )_{s_0} \cap W ( \la )$ alltogether with the irreducible $(\mathfrak{sl} ( 2, 0 ) + \wth)$-modules of dimension $m > n$. In particular, we have even stronger statement $\left(E_0\right)^{n} \bv \ne 0$.

We consider the case of local Weyl modules. Suppose that $\bu \in W ( \lambda,0 )$ is a non-zero vector of weight $\gamma$ with $\left< \alpha_0^{\vee}, \gamma \right> = -n < 0$. Then $\bu$ is the image of some $\bv \in F_m W ( \la )_{s_0} \cap W ( \la )$ with $m > n$ (as we discussed in the above) under the natural projection of
$W ( \lambda)$ onto $W ( \lambda,0 )$. As we know that $\left(E_0\right)^{n-1} \bv \ne 0$ in $W ( \lambda )$, it remains to check that the projection of $\left(E_0\right)^{n-1} \bv$ is not zero in $W ( \lambda,0 )$.

By Definition~\ref{loc-w} the module $W ( \lambda,0 )$ is isomorphic to a quotient of $W ( \lambda)$ by the action of the augmentation ideal $I \subset \C [\A^{(\lambda)}]$. Since the action of $\C [\A^{(\lambda)}]$ on $W ( \la )_{s_0}$ preserves the action of $\mathfrak{sl} ( 2, 0 )$, so is $P \in I$. As this $\C [\A^{(\lambda)}]$-action is free by Theorem \ref{w-free}, we have $P \bv' \in F_m W ( \la )_{s_0} \cap W ( \la )$ for $\bv' \in W ( \la )_{s_0}$ only if $\bv' \in F_m W ( \la )_{s_0}$. As $\left(E_0\right)^{n-1} \bv \in F_m W ( \la )_{s_0} \cap W ( \la )$, it is enough to show that $\left(E_0\right)^{n-1} \bv \not\in I \cdot \left( F_m W ( \la )_{s_0} \cap W ( \la ) \right)$.

We prove this assertion by finding a contradiction. So suppose $\left(E_0\right)^{n-1} \bv = P\cdot \bv'$, where $P \in I$ and $\bv' \in F_m W ( \la )_{s_0} \cap W ( \la )$ is a vector of weight $\gamma + (n-1) \alpha_0$. Here we have $\left< \al_0^{\vee}, \gamma + (n-1) \alpha_0 \right> = n-2 < m-1$, the above consideration shows that this weight appears only in the sum of irreducible $( \mathfrak{sl} ( 2, 0 ) + \wth)$-modules inside $F_m W ( \la )_{s_0} \cap W ( \la )$. Therefore, $\bv'$ belongs to a $( \mathfrak{sl} ( 2, 0 ) + \wth)$-module inside $F_m W ( \la )_{s_0} \cap W ( \la )$. This implies that $ \left(F_0\right)^{n-1} \bv'$ is a non-zero element of $F_m W ( \la )_{s_0} \cap W ( \la )$ and 
$P \cdot \left(F_0\right)^{n-1} \bv'$ is proportional to $\bv$, hence we conclude $\bu=0$. This is a contradiction, and hence we conclude that the projection of $\left(E_0\right)^{n-1} \bv$ is not zero in $W ( \lambda,0 )$ as required.
\end{proof}

\subsection{Ext's in the Parabolic Category ${\mathcal O}$}

Recall that $k$ denotes a non-negative integer.

\begin{prop}\label{alc-comb}
Let $\lambda \in P _+^k$. For each $w, v \in W_{\af}$ such that $\overline{w \circ_k \lambda}, \overline{v \circ_k \lambda} \in P _+$, we have $\overline{w \circ_k \lambda} \le \overline{v \circ_k \lambda}$ if $w \le v$.
\end{prop}

\begin{proof}
Note that $\frac{1}{k+1} \lambda \in \h^*$ belongs to the fundamental alcove (denoted by $A_v^+$ for $v = 0$ in \cite{Lus80}). We have $\frac{1}{k+1} \overline{w \circ_k \lambda} \in A = w A_0^+, \frac{1}{k+1} \overline{v \circ_k \lambda} \in A' = v A_0^+$. By \cite[Lemma 3.6]{Lus80}, we have $d ( A, A_0^+ ) = \ell ( w )$, and $d ( A', A_0^+ ) = \ell ( v )$. By the above identification of $d (\bullet, A_0^+)$ and $\ell ( \bullet )$, we deduce $\ell ( w t_{\gamma} ) = \ell ( w ) + \ell ( t_{\gamma} )$ and $\ell ( v t_{\gamma} ) = \ell ( v ) + \ell ( t_{\gamma} )$ for some large dominant coroot $\gamma$. Now the subward property of the Bruhat order yields $w t_{\gamma} < v t_{\gamma}$, that is equivalent to $A \preceq A'$ (inside the dominant chamber with respect to $W$) by the equivalence of assertions after \cite[Claim 4.14]{Soe97}. Since $\overline{w \circ_k \lambda} - \overline{v \circ_k \lambda} \in Q$, we conclude the result.
\end{proof}

\begin{thm}[Kac-Kazhdan \cite{KK79} Theorem 2, cf. Fiebig \cite{Fie03} \S 3.2]\label{link}
Let $\lambda, \mu \in P_+$. Then, we have
$$\bigoplus_{i \in \Z} \mathrm{Ext}^i_{\wg_{k}} ( L_k ( \lambda ), L _k ( \mu ) ) \neq \{ 0 \}$$
only if $\lambda$ and $\mu$ belongs to the same $W_{\af}$-orbit under the $\circ_k$-action. Moveover, we have $[\mathbf M_k ( w \circ_k \lambda ) : L_k ( v \circ_k \lambda )] \neq 0$ only if $w \le v$. \hfill $\Box$
\end{thm}

For a $\wg_k$-module $M$, let us denote $M^{\#}$ the module that is the dual space $M^{\vee}$ (as a graded vector space) equipped with the $\wg_k$-action twisted by $\Theta$. Then, the $\wg_k$-module $L_k ( \lambda )$ viewed as a $\g [z]$-module is sent to $L_k ( \lambda )$, with its $\g [z]$-module structure given through the automorphism $\Phi$ from the standard $\wg_k$-action on $L_k ( \lambda )$. The same procedure makes $M_k ( \mu ) ^{\#}$ into an injective envelope of $V ( \mu )$ (as $M_k ( \mu )$ is a projective cover of $V ( \mu )$).

\begin{prop}[Shapiro's Lemma]\label{sclem}
Let $V$ be a graded $\g [\xi]$-module $($or a graded $\gI$-module$)$ and let $M$ be a $\wg_k$-module. Then, we have
$$\Ext^i_{\wg_k} ( U_k ( \wg ) \otimes_{U ( \g [\xi] + \C d )} V, M ) \cong \Ext^i_{\g [\xi] + \C d} ( V, M ) \hskip 5mm i \in \Z,$$
$($or the isomorphism obtained by replacing $( \g [\xi] + \C d )$ with $\gI).$
\end{prop}

\begin{proof}
It is straight-forward to see that $U ( \wg )_k$ is a free $U ( \g [ \xi ] )$-algebra and also a free $U ( \gI ) / ( K - k )$-algebra (by the PBW theorem). Hence, Shapiro's lemma implies the results.
\end{proof}

\begin{cor}
Let $\lambda, \mu \in P_+$ such that $\lambda \not> \mu$. Then, we have
$$\Ext^1 _{\wg_k} ( M_k ( \lambda ), L_k ( \mu ) ) = \{ 0 \}.$$
In particular, $M_k ( \lambda )$ is projective in the parabolic category $\cO$  when $\lambda \in P_+^k$.
\end{cor}

\begin{proof}
By Proposition \ref{sclem}, the assertion is equivalent to
$$\Ext^1 _{\g [\xi]} ( V ( \lambda ), L_k ( \mu ) ) = \{ 0 \}.$$
Since $\#$ exchanges an injective resolution of $L_k ( \mu )$ (viewed as $\g [\xi]$-module) and a projective resolution of $L_k ( \mu )$ viewed as $\g [z]$-modules, we deduce that
\begin{equation}
\Ext^1 _{\g [\xi]} ( V ( \lambda ), L_k ( \mu ) ) \cong \Ext^1 _{\g [z]} ( L_k ( \mu ), V ( \lambda ) ).\label{extisom}
\end{equation}
Since $M_k ( \mu )$ is the projective cover of $L_k ( \mu )$ as $\g [z]$-modules, the non-trivial extension in the RHS of (\ref{extisom}) occurs in the highest weights in $M_k ( \mu )$. Thanks to Proposition \ref{alc-comb} and Theorem \ref{link}, we need $\mu < \lambda$ to obtain a non-trivial extension. This proves the first assertion. In view of Theorem \ref{link} and the fact that $P_+^k$ is contained in the fundamental domain of the $\circ_k$-action on (the real part of) $\h^*$, we deduce the latter assertion again by Proposition \ref{alc-comb}.
\end{proof}

Recall that a string $\gb ( 0 )$-module is a finite-dimensional space spanned by $\wth$-eigenvectors $\bv,\,E_0 \bv,\,E_0^2 \bv,\dots,E_0^m \bv$ for some $m \in \Z_{\ge 0}$.

\begin{lem}\label{non-split}
Suppose that $\lambda \in P_+^k$.
Let $V$ be a string $\gb ( 0 )$-module containg weights $( \lambda + k \Lambda_0 )$ and $( s _0 \circ_k \lambda + k \Lambda_0 )$. We inflate to regard it as a $\gI$-module through the surjection $\gI \to ( \gb ( 0 ) + \wth )$.  Then, the maximal $\g$-integrable quotient of $U_k ( \tg ) \otimes _{U ( \gI )} V$ contains the non-trivial extension of $M_k ( \lambda )$ by $M_k ( s _0 \circ_k \lambda )$ as its subquotient.
\end{lem}

\begin{proof}
Note that as $\lambda \in P^+_k$ we have $\lambda - s _0 \circ_k \lambda = m \alpha_0$ for some $m>0$.

By $\bv$ denote a lowest weight vector of $V$, that is, a non-zero vector of the head of $V$ (it is unique up to a scalar). We denote by $\mu$ the $( \h + \C d )$-weight of $\bv$, and set $m_0 := \min \{ n \in \Z_{\ge 0} \mid E_0^n \bv = 0 \}$.

We introduce an increasing $\gb ( 0 )$-module filtration on $V$ defined by
$$F^m = \left< E_0^{m_0 - 1} \bv, E_0^{m_0 - 2} \bv, \dots, E_0^{m_0 - m} \bv \right> \hskip 5mm \text{for} \hskip 5mm m \in \Z_{\ge 0}.$$
Then each adjoint graded factor is spanned by a single vector with trivial action of $E_0$.

Thanks to the exactness of the induction, this filtration produce a filtration on $U_k ( \tg ) \otimes _{U ( \gI )} V$ with its adjoint graded quotients isomorphic to $\mathbf M_k ( \mu +  m\alpha_0)$ for $m=m_0-1,m_0-2,\dots$.
  
We denote by $\Omega$ the Casimir element of $\tg$ (see \cite[Chapter 2]{Kac}). It belongs to a suitable completion of $U_k ( \tg )$ and acts on each highest weight module, particularly on finite successive extensions of $\{ \mathbf M_k ( \mu + m\alpha_0) \}_{m \in \Z_{\ge 0}}$.

For each $\mu \in P \oplus \Z \delta$, the action of $\Omega$ on each $\mathbf M_k ( \mu + m \alpha_0)$ is by a scalar, depending on $m$ as a degree two polynomial on $m$ (see \cite[Chapters 2,7]{Kac}). Also action of $\Omega$ on $\mathbf M_k ( \mu )$ is invariant with respect to the $\circ_k$-action of the affine Weyl group on $P$ that arises from highest weights (see \cite{Kac}, Chapter 7). In particular, the scaling factor of $\Omega$ is the same for $\mathbf M_k ( \lambda )$ and $\mathbf M_k ( s _0 \circ_k \lambda )$ and differs for other factors. We denote the scaling factor of $\Omega$ on $\mathbf M_k ( \lambda )$ by $c$.
  
As the action of $\Omega$ commutes with the action of $\tg$ and our module is a finite successive extension of the modules in which $\Omega$ acts by scalars, the generalized $c$-eigenspace $M$ of $\Omega$ in $U_k ( \tg ) \otimes _{U ( \gI )} V$ is a direct summand. Moreover, by the above eigenvalue analysis, it admits a filtration whose adjoint graded quotients are $\mathbf M_k ( \lambda )$ and $\mathbf M_k ( s _0 \circ_k \lambda )$ . Therefore, $M$ is a trivial or non-trivial extension between $\mathbf M_k ( \lambda )$ and $\mathbf M_k ( s _0 \circ_k \lambda )$.

Note that the $\wth$-weight $(s _0 \circ_k \lambda + k \Lambda_0)$-parts of the both of $\mathbf M_k ( \lambda )$ and $\mathbf M_k ( s _0 \circ_k \lambda )$ consists of highest weights (see Theorem \ref{verma}). Hence, in case $M$ is a trivial extension, the $\wth$-weight $(s _0 \circ_k \lambda + k \Lambda_0)$-part of both modules can not get to a non-trivial vector of weight $\lambda$. But in our case it can be done by the action of $E_0$ in $V$ itself, and hence $M$ cannot be a trivial extension.

There is a unique way to make $V' := \bigoplus_{t = 0}^s V( \la + t \vartheta )$ into a graded $\g [\xi]$-module with a simple head $V( \la + s \vartheta )$, and
$$( \g \otimes \xi ) \cdot V( \la + t \vartheta ) = V( \la + ( t- 1) \vartheta ) \hskip 5mm 0 < t \le s$$
thanks to the PRV theorem. The $\g [\xi]$-module $V'$ contains an $\gI$-submodule $\bigoplus_{t = 0}^s V ( \la + t \vartheta )^{\gn}$ which is a string $\gb ( 0 )$-module. We have
$$U_k ( \tg ) \otimes _{U ( \gI )} V \cong U_k ( \tg ) \otimes_{U ( \g [\xi] )} U ( \g [\xi] ) \otimes _{U ( \gI )} V,$$
and the maximal $\g$-integrable quotient of $U_k ( \tg ) \otimes _{U ( \gI )} V$ is obtained by the induction of the maximal $\g$-integrable quotient of $U ( \g [\xi] ) \otimes _{U ( \gI )} V$. Hence, if we arrange $\mu = \la - s \al_0$, then we have a quotient map $V \longrightarrow \bigoplus_{k = 0}^s V ( \la + k \vartheta )^{\gn}$ as $\gI$-modules, that extends to a $\g [\xi]$-module surjection
$$U ( \g [\xi] ) \otimes _{U ( \gI )} V \longrightarrow V'.$$
This quotient map factors through the $\g$-integrable quotient of $U ( \g [\xi] ) \otimes _{U ( \gI )} V$ (as the RHS is so), and hence we conclude that the both of $V ( \la )$ and $V ( \la + m \vartheta )$ survives in the the $\g$-integrable quotient of $U ( \g [\xi] ) \otimes _{U ( \gI )} V$.

In particular, passing $M$ to its the maximal $\g$-integrable quotient, we obtain a non-trivial extension between $M_k ( \lambda )$ and $M_k ( s _0 \circ_k \lambda )$. Therefore, we conclude the assertion.
\end{proof}

\subsection{The main theorem}

\begin{thm}\label{FC-str}
For each $\lambda \in P_+^k$, the $\g [z]$-module $L_{k} ( \lambda )$ admits a filtration by global Weyl modules.
\end{thm}

\begin{rem}
In fact, our proof of Theorem \ref{FC-str} carries over to the case of the twisted affinization of $\g$ by a straight-forward modification.
\end{rem}

\begin{proof}[Proof of Theorem \ref{FC-str}]
We check the condition in Theorem \ref{crit-filt}. Applying $\#$ to the BGGL-resolution of $L_{k} ( \lambda )$, we deduce:
\begin{equation}
0 \rightarrow L_k ( \lambda )^{\#} \rightarrow M _k ( \lambda )^{\#} \stackrel{d_1^{\#}}{\longrightarrow} M _k ( s_0 \circ_k \lambda )^{\#} \stackrel{d_2^{\#}}{\longrightarrow} \bigoplus_{\mu \in W _{\af} \circ _k \lambda} ( M_k ( \mu )^{\#} )^{\oplus m_2 ( \mu )} \rightarrow \cdots.\label{min-inj}
\end{equation}
By chasing the image of the maps obtained by applying $\mathrm{Hom}_{\g [z]} ( W ( \mu, 0 ), \bullet )$ to (\ref{min-inj}), we deduce
\begin{equation}
\mathrm{Ext} ^1 _{\g [z]} ( L_{k} ( \lambda ), W ( - w_0 \mu, 0 )^* ) \cong \mathrm{Ext} ^1 _{\g [z]} ( W ( \mu, 0 ), L_{k} ( \lambda ) ^{\#} ).\label{ext-dual}
\end{equation}
Here we want to show the vanishing of the LHS of (\ref{ext-dual}) for every $\mu \in P_+$.

By Proposition \ref{sclem}, we have
\begin{equation}
\mathrm{Ext} ^1 _{\g [z]} ( W ( \mu, 0 ), L_{k} ( \lambda )^{\#} ) \cong \mathrm{Ext} ^1 _{\wg_k} ( U_k ( \wg ) \otimes_{U ( \g [\xi] )} W ( \mu, 0 ), L_{k} ( \lambda ) ),\label{frob}
\end{equation}
where the module structure on the RHS is twisted by $\Phi$ (e.g. we just let $\g[\xi]$ act on $W ( \mu, 0 )$ through the isomorphism $\g[\xi] \cong \g [z]$). By using the $\g [\xi]$-module filtration on $W ( \mu, 0 )$ and the exactness of the induction, we deduce that $U ( \wg_k ) \otimes_{U ( \g [\xi] )} W ( \mu, 0 )$ is a finite successive extensions of parabolic Verma modules of level $k$. In view of Proposition \ref{sclem}, we have
\begin{align*}
\mathrm{Ext} ^1 _{\g [\xi] \oplus \C d} ( V ( \gamma, 0 ), L_k ( \lambda )) & \cong \mathrm{Ext} ^1 _{\wg_k} ( M_k ( \gamma ), L_k ( \lambda ) ) \\
\mathrm{Ext} ^1 _{\gI} ( \C_\gamma, L_k ( \lambda )) & \cong \mathrm{Ext} ^1 _{\wg_k} ( \mathbf M_k ( \gamma ), L_k ( \lambda ) ).
\end{align*}
By Theorem \ref{BGGL-res}, and the genuine BGG-resolution, we deduce
\begin{align}\nonumber
\mathrm{Ext} ^1 _{\wg_k} ( M_k ( \gamma ), L_k ( \lambda ) ) \neq \{ 0 \} & \hskip 3mm \Rightarrow \hskip 3mm \gamma = s_0 \circ_k \lambda,
\\
\mathrm{Ext} ^1 _{\wg_k} ( \mathbf M_k ( \gamma ), L_k ( \lambda ) ) \neq \{ 0 \} & \hskip 3mm \Rightarrow \hskip 3mm  \gamma \in \{ s_i \circ_k \lambda \}_{i \in \mathtt I_\af},
\label{link-lambda}
\end{align}
and these extensions are at most one-dimensional. It follows that

\begin{equation*}
\xymatrix{
\mathrm{Ext} ^1 _{\g [\xi] \oplus \C d} ( V ( \gamma, 0 ), L_k ( \lambda )) \ar@{=}[r]\ar@{^{(}->}[d] & \mathrm{Ext} ^1 _{\wg_k} ( M_k ( \gamma ), L_k ( \lambda ) )\\
\mathrm{Ext} ^1 _{\gI} ( \C_\gamma, L_k ( \lambda )) \ar@{=}[r] & \mathrm{Ext} ^1 _{\wg_k} ( \mathbf M_k ( \gamma ), L_k ( \lambda ) )
}
\end{equation*}
and the extension is at most one-dimensional.

As $\lambda \in  P_+^k$, the classical weights $\overline{\lambda}$ and $\overline{s_0 \circ_k \lambda}$ are both dominant.

We have the following map
\begin{equation*}
\xymatrix{
W ( \mu, 0 ) \ar[r]^{E_0^{m}} & W ( \mu, 0 )\\
\ar@{^{(}->}[u] W ( \mu, 0 )^{\overline{\la}} \ar[r]_{E_0^{m}} & \ar@{^{(}->}[u] W ( \mu, 0 )^{\overline{s_0 \circ_k \la}}
} \hskip 3mm m := - \left< \vartheta^{\vee}, \lambda \right> + k \ge 0,
\end{equation*}
where $W ( \mu, 0 )^{\gamma}$ for $\gamma \in P$ denotes the $\h$-weight $\gamma$-part of $W ( \mu, 0 )$.

By Lemma \ref{glWD}, the map $E_0^m$ in the bottom line is injective. Therefore, Lemma \ref{non-split} implies that the extension of $M_k ( s_0 \circ_k \lambda )$ by $M_k ( \lambda )$ is already attained in the subquotient of $U_k ( \tg ) \otimes_{U ( \gI )} W ( \mu, 0 )$.

This forces the RHS of (\ref{frob}) to be zero for the parabolic Verma modules arising from $W ( \mu, 0 )^{\gamma}$ for $\gamma = s_0 \circ_k \lambda$. In view of (\ref{link-lambda}), we deduce
$$\mathrm{Ext} ^1 _{\wg_k} ( U_k ( \tg ) \otimes_{U ( \g [\xi] + \wth )} W ( \mu, 0 ), L_{k} ( \lambda ) ) = \{ 0 \} \hskip 5mm \forall \mu \in P_+,$$
that implies the result.
\end{proof}

\begin{cor}
Let $\lambda \in P_+^k$. The projective resolution of $L_{k} ( \lambda )$  provided by Theorem \ref{BGGL-res} respects the filtration by global Weyl modules.
\end{cor}

\begin{proof}
Applying Theorem \ref{crit-filt} to the long exact sequence obtained by applying $\Ext^{\bullet}_{\g [z]} ( \bullet, W (\mu, 0)^* )$ to the short exact sequence
$$0 \rightarrow \ker_0 \rightarrow M_k ( \lambda ) \to L_k ( \lambda ) \to 0,$$
we deduce $\ker_0$ admits a filtration by global Weyl modules. 

We have $\ker \, d_i \cong \mathrm{Im} \, d_{i+1}$ for each $i \ge 0$, and it is finitely generated by examing the next term. Hence, we apply the same argument by replacing $M_k ( \lambda )$ with $\bigoplus_{w \in W \backslash W_{\af}, \ell ( w ) = i} M_k ( w \circ_k \lambda )$ and $L_k ( \lambda )$ with $\ker \, d_i$ to deduce the assertion inductively.
\end{proof}

\section{The level-restricted Kostka polynomials}\label{lrm}

We work in the setting of the previous section. Fix a positive integer $k$ in the sequel. Let $M$ be a $\g$-integrable graded $\g [z]$-module. Then, we define its $i$-th relative homology group as:
$$H_i (\g [z], \g; M ) := \Hom_{\g} ( \C, \mathbb R^{-i} \Hom_{\g [z]_1} ( M, \C ) ).$$
This is a graded vector space.

\begin{rem}\label{hom-rem}
Thanks to Theorem \ref{CG}, each term of the projective resolution of $M$ (as $\g [z]_1$-modules) can be taken to be $\g$-integrable. In particular, $\mathbb R^{-i} \Hom_{\g [z]_1} ( M, \C )$ is semi-simple as $\g$-modules and hence our definition of the relative homology group coincides with these in \cite[Chapter I]{BW00}.
\end{rem}

\begin{lem}\label{inf-Kos}
Let $k \in \Z_{> 0}$ and let $\lambda, \mu \in P_+$. We have
$$\mathsf{gdim} \, H_0 ( \g [z], \g; W ( \mu, 0 ) \otimes_\C M_{k} ( \lambda ) ) = [W ( \mu, 0 )^* : V ( \lambda ) ]_q.$$
\end{lem}

\begin{proof}
By unwinding the definition and applying Theorem \ref{CG}, we have
\begin{align*}
H_0 ( \g [z], \g; W ( \mu, 0 ) \otimes_\C M_{k} ( \lambda ) ) & = \Hom_{\g} ( \C, \mathbb \Hom_{\g [z]_1} ( W ( \mu, 0 ) \otimes_\C M_{k} ( \lambda ), \C ) )\\
\cong \Hom_{\g [z]} ( M_{k} ( \lambda ), W ( \mu, 0 )^* ) & = \Hom_{\g [z]} ( P ( \lambda ), W ( \mu, 0 )^* ).
\end{align*}
Therefore, the assertion holds.
\end{proof}

\begin{defn}[Feigin and Feigin \cite{FF05}]
Let $\lambda \in P_+^k$ and $\mu \in P_+$. We define the level-restricted Kostka polynomial $P^{(k)} _{\mu, \lambda} ( q ) \in \Z [q,q^{-1}]$ by
$$P^{(k)} _{\mu, \lambda} ( q^{-1} ) := \mathsf{gdim} \, H_0 ( \g [z], \g; W ( - w_0 \mu, 0 ) \otimes_\C L_{k} ( \lambda ) ).$$
\end{defn}

\begin{rem}
{\bf 1)} The original definition of the level-restricted Kostka polynomial is due to Schilling-Warnaar \cite{SW99}. We provide a comparison with a straight-forward generalization of their definition with ours in Corollary \ref{comp-mac}. {\bf 2)} Since $P ( \lambda ) \cong \varprojlim_{k \to \infty} L_k ( \lambda )$ as graded $\g [z]$-modules, we have a well-defined limit
$$ P _{\mu, \lambda} ( q ) = \lim_{k \to \infty} P^{(k)}_{\mu, \lambda} ( q ).$$
{\bf 3)} If $\g$ is of type $\mathsf{A}$, then $P_{\mu, \lambda} ( q )$ coincides with the Kostka polynomial if we take ``transpose" of $\mu$.
\end{rem}

\begin{thm}\label{kmain}
Let $\lambda \in P_+^k$ and $\mu \in P_+$. We have
$$H_i ( \g [z], \g; W ( \mu, 0 ) \otimes_\C L_{k} ( \lambda ) ) = \{ 0 \} \hskip 3mm i \neq 0.$$
\end{thm}

\begin{proof}
Taking account into Remark \ref{hom-rem}, we have
\begin{align*}
H_i ( \g [z], \g; W ( \mu, 0 ) \otimes_\C L_{k} ( \lambda ) ) & = \Hom_{\g} ( \C, \mathbb R^{-i} \Hom_{\g [z]_1} ( W ( \mu, 0 ) \otimes_\C L_{k} ( \lambda ), \C ) )\\
& \cong \Hom_{\g} ( \C, \mathbb R^{-i} \Hom_{\g [z]_1} ( L_{k} ( \lambda ), W ( \mu, 0 )^* ) )\\
& \cong \mathbb R^{-i} \Hom_{\g [z] \mathchar`-\mathsf{mod}} ( L_{k} ( \lambda ), W ( \mu, 0 )^* )\\
& \cong \Ext^{-i}_{\g [z] \mathchar`-\mathsf{mod}} ( L_{k} ( \lambda ), W ( \mu, 0 )^* ).
\end{align*}
By Theorem \ref{FC-str} and Theorem \ref{crit-filt} 2), we deduce that
$$\Ext^{-i}_{\g [z] \mathchar`-\mathsf{mod}} ( L_{k} ( \lambda ), W ( \mu, 0 )^* ) = \{0\} \hskip 3mm \text{for each} \hskip 2mm - i \neq 0$$
as required.
\end{proof}

\begin{cor}\label{L-rec}
For each $\lambda \in P_+^k$ and $\mu \in P_+$, we have
$$P ^{(k)}_{\mu, \lambda} ( q ) = ( L_k ( \lambda ) : W ( \mu ) )_q.$$
\end{cor}

\begin{proof}
By Theorem \ref{FC-str}, we can repeatedly apply (the $\mathrm{Ext}^1$-part of) Theorem \ref{crit-filt} to short exact sequences that respects the filtration by the global Weyl modules $L_k ( \lambda )$. This yields the additivity of the $\Hom$-part, namely
$$\mathsf{gdim} \, \Hom_{\g [z]} ( L_{k} ( \lambda ), W ( \mu, 0 )^* ) = \sum_{\gamma \in P_+} \overline{( L_{k} ( \lambda ) : W ( \gamma ) )_{q}} \cdot \mathsf{gdim} \, \Hom_{\g [z]} ( W ( \gamma ), W ( \mu, 0 )^* ).$$
Applying the $\mathrm{Hom}$-part of Theorem \ref{crit-filt}, we conclude the result.
\end{proof}

The following is the Teleman's Borel-Weil-Bott theorem \cite{Tel95}, that we supply a proof along the line of Theorem \ref{kmain} (without assuming materials in \S \ref{rep-current}).

\begin{cor}[Teleman \cite{Tel95}]\label{TBWB}
For each $\la, \mu \in P_+^k$ and $w \in W \backslash W_\af$, we have
$$\mathsf{gdim} \, H_i ( \g [z], \g; V ( \overline{w \circ_k \mu}, 0 )^* \otimes_\C L_{k} ( \lambda ) ) = \begin{cases} q^{\left< d, w \circ_k \lambda \right>} & (\la = \mu, i = - \ell ( w ) ) \\ 0 & (otherwise) \end{cases}.$$
\end{cor}

\begin{proof}
Note that we have $\overline{w \circ_k \mu} \in P_+$. By unwinding the definition as in the proof of Theorem \ref{kmain}, we have
$$
H_i ( \g [z], \g; V ( \overline{w \circ_k \mu}, 0 )^* \otimes_\C L_{k} ( \lambda ) )
\cong \Ext_{\g [z]}^{-i} ( L_{k} ( \lambda ), V ( \overline{w \circ_k \mu} ) ).
$$
As the BGGL resolution is minimal as a projective resolution, we conclude the result.
\end{proof}

For each $\mu \in P_{+}$, let us regard $W ( \mu, 0 )$ as a $\g [\xi]$-module through $\Phi$. We define
$$\mathbf W_k ( \mu ) := U_k ( \widehat{\g} ) \otimes_{U ( \g [\xi] )} W ( \mu, 0 )$$
and call it the generalized Weyl module of $\wg$ with highest weight $\mu$ and level $k$. We fix a sequence of distinct points $\vec{a} = (a_1,\ldots,a_m) \in \C^m$ and weights $\vec{\mu} = ( \mu_1, \mu_2, \ldots, \mu_m ) \in ( P_+ )^m$ such that $\mu = \mu_1 + \mu_2 + \cdots + \mu_m$. Then, we define the space of generalized conformal coinvariants as:
$$CC ( \vec{a}, \vec{\mu}, \lambda ) := \frac{L_k ( \lambda) \otimes \bigotimes_{i = 1} ^m \mathbf W_k ( \mu_i )}{U ( \g \otimes _\C \mathcal O _{\vec{x}} ) \left( L_k ( \lambda ) \otimes \bigotimes_{i = 1} ^m \mathbf W_k ( \mu_i ) \right)},$$
where $\mathcal O_{\vec{x}} = \C [x, \frac{1}{x - a_1}, \ldots, \frac{1}{x - a_m}]$ and an element $X \otimes f \in \g \otimes \mathcal O_{\vec{x}}$ acts on $\mathbf W_k ( \mu_i )$ ($1 \le i \le m$) through its Laurent expansion by $z = x - a_i$, and acts on $L_k ( \lambda)$ through its Laurent expansion by $z = 1/x$ (cf. Teleman \cite[\S 3.6]{Tel95}).

\begin{thm}
Let $\lambda \in P_+^k$ and let $\mu \in P_{+}$. The space
$$H_0 ( \g [z], \g; W ( \mu ) \otimes_\C L_{k} ( \lambda ) )^{\vee}$$
is a free $\C [\A^{(\mu)}]$-module, and the specialization to $\vec{a} \in \A^{(\mu)}$ corresponding to distinct points $a_1, a_2,\ldots,a_m$ with their multiplicities $\mu_1,\mu_2,\ldots, \mu_m \in P_+$ $($i.e. when the multiplicity of $a_k$ with respect to the $i$-th set of unordered points is $\left< \alpha_i ^{\vee}, \mu_k \right>)$ yields an isomorphism of vector spaces
$$\C _{\vec{a}} \otimes _{\C [\A^{(\mu)}]} H_0 ( \g [z], \g; W ( \mu ) \otimes_\C L_{k} ( \lambda ) )^{\vee} \cong CC ( \vec{a}, \vec{\mu}, \lambda ).$$
\end{thm}

\begin{proof}
By Theorem \ref{free}, the module $H_0 ( \g [z], \g; W ( \mu ) \otimes_\C L_{k} ( \lambda ) )$ admits an increasing exhausting filtration whose adjoint graded quotient is the direct sum of submodules of $H_{0} ( \g [z], \g; W ( \mu, 0 ) \otimes_\C L_{k} ( \lambda ) )$ with grading shifts. By Theorem \ref{kmain}, we deduce that the adjoint graded quotient of $H_0 ( \g [z], \g; W ( \mu ) \otimes_\C L_{k} ( \lambda ) )$ is the direct sum of $H_0 ( \g [z], \g; W ( \mu, 0 ) \otimes_\C L_{k} ( \lambda ) )$ (instead of its proper submodules) with grading shifts. The dual of our homology group commutes with inverse limit by the degree-wise Mittag-Leffler condition thanks to \cite[Theorem 0]{Tel95} and Theorem \ref{kmain}. Therefore, we conclude that the free $\C [\A^{(\mu)}]$-action on $W ( \mu )$ lifts to $H_0 ( \g [z], \g; W ( \mu ) \otimes_\C L_{k} ( \lambda ) )^{\vee}$. In addition, we have
$$H_i ( \g [z], \g; W ( \mu ) \otimes_\C L_{k} ( \lambda ) ) = \{ 0 \} \hskip 5mm i \neq 0.$$

By the semi-continuity theorem applied to $0 \in \A^{(\mu)}$ by regarding $W ( \mu ) \otimes_\C L_{k} ( \lambda )$ as a finitely generated $U ( \g [z] ) \otimes \C [\A^{(\mu)}]$-module (through the above construction), we deduce that the natural map
$$\C _{\vec{a}} \otimes _{\C [\A^{(\mu)}]} H_0 ( \g [z], \g; W ( \mu ) \otimes_\C L_{k} ( \lambda ) )^{\vee}  \rightarrow H_0 ( \g [z], \g; \C _{\vec{a}} \otimes _{\C [\A^{(\mu)}]} W ( \mu ) \otimes_\C L_{k} ( \lambda ) )^*$$
is an isomorphism. The RHS is further isomorphic to
$$H_0 ( \g [z], \g; \bigotimes_{i = 1}^{m} W ( \mu_{i}, a_{i} ) \otimes_\C L_{k} ( \lambda ) )^*,$$
where $\{a_{1}, \ldots, a_{m} \} \subset \C$ denotes the set of distinct points in the configuration $\vec{a}$, and $\mu_{i}$ is the sum of fundamental weights that is supported on $a_{i}$. This last vector space is precisely $CC ( \vec{a}, \vec{\mu}, \lambda )$ in view of \cite[Corollary 3.6.6]{Tel95}.
\end{proof}

\begin{thm}[Weyl-Kac character formula \cite{Kac}, see also \cite{Kum02} Theorem 2.2.1]\label{KT}
For each $\lambda \in P_+^k$, we have the following equality of characters:
$$
\mathsf{ch} \, L _k ( \lambda ) = \sum_{w \in W \backslash W _{\af}} (-1)^{\ell ( w )} \mathsf{ch} \, M_k ( \overline{w \circ _k \lambda} ) \otimes _{\C} \C _{\left< d, w \circ_k \lambda \right> \delta}.
$$
\end{thm}

\begin{cor}\label{altP}
For each $\lambda \in P^k_+$ and $\mu \in P_+$, we have
$$P _{\mu, \lambda}^{(k)} ( q ) = \sum_{w \in W \backslash W_{\af}} (-1)^{\ell ( w )} q^{- \left< d, w \circ_k \lambda \right>} P _{\mu, \overline{w \circ_k \lambda}} ( q ).$$
\end{cor}

\begin{proof}
Taking into account the fact that $\{ \mathsf{ch}\, W(\lambda) \}_{\lambda \in P^+}$ forms a $\C (\!(q)\!)$-basis in the space of graded characters, the expansion coefficients of characters of $L_{k} ( \la )$ and $M_k ( w \circ_k \la )$ in terms of $\{ \mathsf{ch}\, W(\lambda) \}_{\lambda \in P^+}$ are determined uniquely. They are $P _{\mu, \lambda}^{(k)} ( q )$ and $P _{\mu, \overline{w \circ_k \lambda}} ( q )$ by Corollary \ref{L-rec} and Lemma \ref{inf-Kos}, respectively. They obey the linear relation coming from Theorem \ref{KT}, that implies the result.
\end{proof}

\section{The Feigin realization of global Weyl modules}
We retain the setting of the previous section. We assume that $\g$ is of type $\mathsf{ADE}$ in addition.

\begin{thm}\label{mult-one}
Let $\varpi \in P_+^1$ and let $\mu \in P_+$. Then, we have
$$( L_{1} ( \varpi ) : W ( \mu ) )_q = \begin{cases} q ^{- \left< d, w ( \varpi + \Lambda_0 ) \right>} & (\mu = \overline{w ( \varpi + \Lambda_0 )}, w \in W_{\af})\\ 0 & (\text{otherwise})\end{cases}.$$
\end{thm}

\begin{proof}
By Cherednik-Feigin \cite[(1.25)]{CF13}, the character of $L_{1} ( \varpi )$ is the multiplicity-free sum of these of $W ( \mu )$ up to grading shifts (note that we used $Q = Q^{\vee} \subset W_\af$ here). Since the graded characters of $\{W ( \mu )\}_{\mu \in P_+}$ are linearly independent, we deduce the $q=1$ case of the assertion.

Let us normalize $L_1 ( \varpi )$ such that its highest weight vector has $d$-degree $0$. For $w \in W_{\af} $ we set $a_w := \left< d, w ( \varpi + \Lambda_0 ) \right>$. Then, we have
\begin{claim}\label{first}
Assume that $w \in W_\af$. Let $\mu$ be the $\h$-weight of $L_1 ( \varpi )$ that appears in the $d$-degree $> a_w$-part. Then, we have $\mu < \overline{w ( \varpi + \Lambda_0 )}$.
\end{claim}

\begin{proof}
Easy consequence of \cite[Proposition 11.3]{Kac}
\end{proof}

We return to the proof of Theorem \ref{mult-one}. In view of the $q=1$ case, Claim \ref{first} forces the degree shift of $W ( \overline{w ( \varpi + \Lambda_0 )} )$ appearing in $L _1 ( \varpi )$ to be exactly $a_w$ as required.
\end{proof}

Let us define the thick affine Grassmanian as
$$\mathbf{Gr} _G := G (\!( \xi )\!) / G [z],$$
that can be presented as an infinite type scheme \cite{Kas89,Kat17b}. The scheme $\mathbf{Gr} _G$ carries a canonical $G_{\mathrm{sc}} (\!( \xi )\!)$-equivariant line bundle $\cO ( 1 )$ that we call the determinant line bundle (see e.g. Kashiwara \cite{Kas90}). 
%For each $\lambda \in P$ and $k \in \Z$, the manifold $\mathbf{Fl}_G$ carries a $\widetilde{G}_{\mathrm{sc}} (\!( \xi )\!)$-equivariant line bundle $\cO_{\mathbf{Fl}_G} ( \lambda + k \Lambda_0 )$. It is dominant if and only if $\lambda \in P_+^k$ (see e.g. Kashiwara \cite{Kas90}).

By our adjointness assumption on $G$, the set of connected components of the scheme $\mathbf{Gr}_G$ is in bijection with the set of level one fundamental weights $P_+^1$ (see Zhu \cite[\S 0.2.5]{Zhu09} but beware that our affine Grassmanian is ``thick"). For each $\varpi \in P_+^1$, we denote by ${}^{\varpi} \mathbf{Gr}_G$ the corresponding component. 
%We have a natural projection $\pi : \mathbf{Fl}_G \rightarrow {}^{\varpi} \mathbf{Gr}_G$.

Each $\lambda \in P_+$ defines a cocharacter of $H$ by our assumptions of $G$. Hence, it defines a point of $H(\!(\xi)\!)$, and hence a point $[\xi^{\la}] \in \mathbf{Gr}_G$. From this, we define the Schubert variety $\mathbf{Gr}_G^{\lambda}$ as the $G [\![\xi]\!]$-orbit through $[\xi^{\la}]$. We have $\mathbf{Gr}_G^{\lambda} \subset \overline{\mathbf{Gr}_G^{\mu}}$ if and only if $\lambda \ge \mu$ (see e.g. Kashiwara-Tanisaki \cite[\S 1.2]{KT95}).

\begin{thm}[\cite{Kat17b} Theorem 2.17 and Theorem 2.19]\label{KSconj}
Let $\varpi \in P_+^1$ and let $\lambda \in P_+$ such that $\varpi \le \lambda$, we have
$$\Gamma (\mathbf{Gr}_G^{\lambda}, \cO ( 1 ) )^\vee \cong U ( \g [z] ) \bv_{\lambda} \subset L_1 ( \varpi ),$$
where $\bv_{\lambda}$ is an extremal weight vector of $L_1 ( \varpi )$ of $\h$-weight $\lambda$. Moreover, the natural restriction map
$$\Gamma ( {}^{\varpi} \mathbf{Gr}_G, \cO ( 1 ) ) \rightarrow \Gamma ( \mathbf{Gr}_G^{\lambda}, \cO ( 1 ) )$$
is surjective. \hfill $\Box$
\end{thm}

\begin{thm}\label{Fmain}
Let $\varpi \in P_+^1$ and let $\lambda \in P_+$ such that $\varpi \le \lambda$. We have an isomorphism
$$W ( \lambda )^\vee \cong \ker \, \left( \Gamma (\mathbf{Gr}_G^{\lambda}, \cO ( 1 ) ) \longrightarrow \bigoplus_{\mu > \lambda} \Gamma (\mathbf{Gr}_G^{\mu}, \cO ( 1 ) )\right)$$
as $\g [z]$-modules. $($Here the maps in the RHS are the restriction maps.$)$
\end{thm}

\begin{proof}
We borrow the setting in the proof of Theorem \ref{mult-one}. By Claim \ref{first}, the extremal weight vector of $L_k ( \varpi )$ with its $\wth$-weight $w ( \varpi + \Lambda_0 )$ is contained in the head of $W ( \overline{w ( \varpi + \Lambda_0 )} )$ in the $\g [z]$-module stratification of $L_1 ( \varpi )$ by global Weyl modules. Hence, Theorem \ref{KSconj} implies that the graded $\g [z]$-module
$$\Gamma_c (\mathbf{Gr}_G^{\lambda}, \cO ( 1 ) )^\vee := \ker \, \left( \Gamma (\mathbf{Gr}_G^{\lambda}, \cO ( 1 ) ) \longrightarrow \bigoplus_{\mu > \lambda} \Gamma (\mathbf{Gr}_G^{\mu}, \cO ( 1 ) )\right) ^\vee$$
admits a surjection to $W ( \lambda )$, that we denote by $\kappa_{\lambda}$. Thus, we have
\begin{equation}
\sum_{\lambda \in \varpi + Q_+ \cap P_+} q^{-\left< d, w_{\la} ( \varpi + \Lambda_0 ) \right>} \mathsf{ch} \, W ( \lambda ) \le \sum_{\lambda \in \varpi + Q_+ \cap P_+} \mathsf{ch} \, \Gamma_c (\mathbf{Gr}_G^{\lambda}, \cO ( 1 ) )^\vee = \mathsf{ch} \, L_1 ( \varpi ),\label{char-est-f}
\end{equation}
where $w_\la \in W_\af$ denote the element so that $\overline{w_{\la} ( \varpi + \Lambda_0 )} = \la$ and $\le$ means we have the inequality for every coefficients of the monomial in $P \times \Z \delta$. By construction, the inequality in (\ref{char-est-f}) is a genuine inequality if and only if $\kappa_{\lambda}$ fails to be surjective for some $\lambda \in ( \varpi + Q_+ \cap P_+ )$. By Theorem \ref{mult-one}, the third term of (\ref{char-est-f}) is equal to the first term of (\ref{char-est-f}). Therefore, we deduce that every $\kappa_{\lambda}$ induces an isomorphism
$$\Gamma_c (\mathbf{Gr}_G^{\lambda}, \cO ( 1 ) )^\vee = \ker \, \left( \Gamma (\mathbf{Gr}_G^{\lambda}, \cO ( 1 ) ) \longrightarrow \bigoplus_{\mu > \lambda} \Gamma (\mathbf{Gr}_G^{\mu}, \cO ( 1 ) )\right) ^\vee \stackrel{\cong}{\longrightarrow} W ( \lambda )$$
as required.
\end{proof}

\section{A combinatorial definition of level restricted Kostka polynomials}\label{comb-mac}

This section exhibits a collection of folklore statements that are (most likely) originally due to Okado \cite{O}. Thus, we do not claim the novelty of the materials here. However, we provide their proofs that maybe of independent interest.

We employ the setting of \S \ref{aff-Lie}. We have quantum algebras $U_q ( \tg )$ and $U_q ( \wg )$ associated to $\tg$ and $\wg$ such that $U_q ( \wg ) \subset U_q ( \tg )$ (cf. \cite[\S 3.1]{HK02}).

\begin{defn}[Crystals, cf. \cite{HK02} \S 4.2]
An affine crystal
$$\mathbb B = ( B, \mathrm{wt}, \{ \varepsilon_i, \psi_i, \tilde{e}_i, \tilde{f}_i \}_{i \in \tI_{\af}} )$$
consists of the following data:
\begin{enumerate}
\item $B$ is a set, $\mathrm{wt} : B \rightarrow \widetilde{P}$, $\varepsilon_i : B \rightarrow \Z_{\ge 0}$ ($i \in \tI_{\af}$), and $\psi_i : B \rightarrow \Z_{\ge 0}$ ($i \in \tI_{\af}$) are maps;
\item For $b \in B$ and $i \in \tI_\af$, we have $\left< \alpha_i^{\vee}, \mathrm{wt} \, b \right> = - \varepsilon _i ( b ) + \psi_i ( b )$;
\item For $b \in B$ and $i \in \tI_\af$, we have maps $\tilde{e}_i : B \rightarrow B \sqcup \{\emptyset\}$ and $\tilde{f}_i : B \rightarrow B \sqcup \{\emptyset\}$ with the following properties:
\begin{itemize}
\item We have $\widetilde{e}_i ( b ) \neq \emptyset$ if and only if $\varepsilon_i ( b ) > 0$. We have $\widetilde{f}_i ( b ) \neq \emptyset$ if and only if $\psi_i ( b ) > 0$;
\item If $\varepsilon_i ( b ) > 0$, then $\mathrm{wt} \, \widetilde{e}_i ( b ) = \mathrm{wt} \, b + \alpha_i$. If $\psi_i ( b ) > 0$, then $\mathrm{wt} \, \widetilde{f}_i ( b ) = \mathrm{wt} \, b - \alpha_i$;
\item If $\varepsilon_i ( b ) > 0$, then $\widetilde{f}_i ( \widetilde{e}_i ( b ) ) = b$. If $\psi_i ( b ) > 0$, then $\widetilde{e}_i ( \widetilde{f}_i ( b ) ) = b$.
\end{itemize}
\end{enumerate}
A classical crystal $\mathbb B = ( B, \mathrm{wt}, \{ \varepsilon_i, \psi_i, \tilde{e}_i, \tilde{f}_i \}_{i \in \mathtt I_{\af}} )$ is the data obtained from the definition of an affine crystal by replacing $\widetilde{P}$ with $\widehat{P}$. We refer crystal as either affine or classical crystal, and refer $B$ as its underlying set. By abuse of notation, we may abbreviate $b \in B$ by $b \in \mathbb B$. A morphism of a crystal is a map of underlying set that intertwines $\mathrm{wt}, \{ \varepsilon_i, \psi_i, \tilde{e}_i, \tilde{f}_i \}_{i \in \tI_{\af}}$ (that is usually referred to as a strict morphism in the literature). \\
A highest weight element (resp. finite highest weight element) in a crystal $\mathbb B$ is an element $b$ such that $\varepsilon_i ( b ) = 0$ for every $i \in \tI_{\af}$ (resp. every $i \in \tI$). A crystal is connected if and only if each two elements are connected by finitely many sequences of $\tilde{f}_i$ and $\tilde{e}_i$. For an affine crystal $\mathbb B$ and $i \in \tI_\af$, an $i$-string is a subgraph $B' \subset B$ that is closed (and connected) under the action of $\{ \tilde{e}_i, \tilde{f}_i \}$. For $b \in B$ and $i \in \tI_\af$, we denote the $i$-string that contains $b$ as $S_i ( b )$.
\end{defn}

\begin{defn}[Tensor product of crystals, cf. \cite{HK02} \S 4.4]\label{crys-tensor}
Let $\mathbb B_1$ and $\mathbb B_2$ be crystals with underlying sets $B_1, B_2$. We define the tensor product crystal $\mathbb B_1 \otimes \mathbb B_2$ with its underlying set $B_1 \times B_2$ (to which we refer its element $(b_1,b_2)$ as $b_1 \otimes b_2$) by defining: 
\begin{itemize}
\item For each $b_1 \in B_1$ and $b_2 \in B_2$, we have $\mathrm{wt} \, ( b_1 \otimes b_2 ) := \mathrm{wt} \, b_1 + \mathrm{wt} \, b_2$;
\item For each $b_1 \in B_1$, $b_2 \in B_2$, and $i \in \tI_\af$, we define
$$\hskip -10mm \tilde{e}_i ( b_1 \otimes b_2 ) := \begin{cases} \tilde{e}_i b_1 \otimes b_2 & (\psi_i (b_1) \ge \varepsilon_i (b_2)) \\ b_1\otimes \tilde{e}_i b_2 & (\psi_i (b_1) < \varepsilon_i (b_2)) \end{cases}, \text{ and } \tilde{f}_i ( b_1 \otimes b_2 ) := \begin{cases} \tilde{f}_i b_1 \otimes b_2 & (\psi_i (b_1) > \varepsilon_i (b_2)) \\ b_1 \otimes \tilde{f}_i b_2 & (\psi_i (b_1) \le \varepsilon_i (b_2)) \end{cases},$$
where we understand that $\emptyset \otimes b_2 = b_1 \otimes \emptyset = \emptyset$.
\end{itemize}
The functions $\varepsilon_i$ and $\psi_i$ ($i \in \tI_\af$) are uniquely determined by the above.
\end{defn}

For a crystal $\mathbb B$, we define its character as:
$$\mathsf{ch} \, \mathbb B := \sum_{b \in \mathbb B} e^{\mathrm{wt} \, ( b )}.$$

\begin{thm}[Kashiwara \cite{Kas91}, cf. \cite{HK02} \S 5.1]\label{Blam}
For each $\Lambda \in \widetilde{P}_+$ such that $\left< K, \Lambda \right> = k \in \Z_{> 0}$, we have an affine crystal $\mathbb B ( \Lambda )$ that parametrizes a basis of $L_k ( \overline{\Lambda} )$ such that $\mathsf{ch} \, \mathbb B ( \Lambda ) = \mathsf{ch} \, L_k ( \overline{\Lambda} )$ up to $e^{m \delta}$-twist for some $m \in \Z$. Each $\mathbb B ( \Lambda )$ contains a unique element $b_{\Lambda}$ such that $\mathrm{wt} \, b_{\Lambda} = \Lambda$ and $\widetilde{e}_{i} b_{\Lambda} = \emptyset$ for every $i \in \tI_\af$.
\end{thm}

\begin{thm}[Kashiwara \cite{Kas93} \S 3]\label{Demazure}
For each $\Lambda \in \widetilde{P}_+$ and $w \in W_\af$, we have a subset $\mathbb B ( \Lambda )_w \subset \mathbb B ( \Lambda )$ that is stable under the action of $\widetilde{e}_{i}$ for every $i \in \tI_\af$ with the following properties:
\begin{enumerate}
\item For each $i \in \tI_\af$ such that $s_i w > w$, we have
$$\mathbb B ( \Lambda )_{s_i w} = \bigcup_{m \ge 0} \widetilde{f}_i^m \mathbb B ( \Lambda )_{w};$$
\item For each $b \in \mathbb B ( \Lambda )_{w}$, the set
$$\{ \widetilde{e}_i^n \widetilde{f}_i^m b \}_{n, m \ge 0} \cap \mathbb B ( \Lambda )_{w}$$
is either singleton or an $i$-string.
\end{enumerate}
\end{thm}

For each $\Lambda \in \widetilde{P}$ and $i \in \tI_\af$, we define
$$\chi ( e^{\Lambda} ) := \frac{\sum_{w \in W_\af} ( -1 )^{\ell ( w )} e^{w \circ \Lambda}}{\prod_{\alpha \in \Delta_\af^-} ( 1 - e^{\alpha} )^{\mathrm{mult} \, \alpha}} \hskip 3mm \text{and}, \hskip 3mm D_i ( e^{\Lambda} ) := \frac{e^{\Lambda} - e^{s_i \circ \Lambda}}{( 1 - e^{-\alpha_i} )}$$
where $\mathrm{mult} \, \alpha$ denote the dimension of the $\alpha$-root space of $\widetilde{\g}$. We define
$$\Z [\widetilde{P}]^{\wedge} := \varprojlim_n \Z [\widetilde{P}] / (e^{\beta} \mid \beta \in \widetilde{P}, \left< d, \beta \right> < - n).$$
The original form of the Weyl-Kac character formula asserts (see Theorem \ref{KT}) that
\begin{equation}
\chi ( e^{\Lambda} ) = \mathsf{ch} \, L_k ( \Lambda ) \hskip 5mm \Lambda \in \widetilde{P}_+. \label{WCF}
\end{equation}
In addition, $\chi$ naturally extend to a linear operator
$$\Z [\widetilde{P}] \rightarrow \Z [\widetilde{P}]^{\wedge},$$
while $D_i$ ($i \in \tI_\af$) define linear operators on $\Z [\widetilde{P}]$. By Kumar \cite[Theorem 8.2.9 and \S 8.3]{Kum02}, there exists an infinite sequence
$$\mathbf i = (i_1,i_2,\ldots) \in \tI_\af^\infty$$
such that
\begin{equation}
\chi = \lim_{k\to\infty} D_{i_k} \circ \cdots D_{i_2} \circ D_{i_1} \hskip 5mm \text{on} \hskip 3mm \Z [\widetilde{P}]. \label{WCFs}
\end{equation}

\begin{defn}[Restricted paths]
Let $\mathbb B$ be a classical crystal. For each $\Gamma \in \widehat{P}_+$, we define the set of restricted paths as:
$$\mathcal P ( \mathbb B, \Gamma ) := \{ b \in \mathbb B \mid \varepsilon_i ( b ) \le \left< \alpha_i^{\vee}, \Gamma \right> \hskip 2mm \text{for each} \hskip 2mm i \in \tI_\af \}.$$
Similarly, for each $\gamma \in P_+$, we define the set of finitely restricted paths as:
$$\mathcal P_0 ( \mathbb B, \gamma ) := \{ b \in \mathbb B \mid \varepsilon_i ( b ) \le \left< \alpha_i^{\vee}, \gamma \right> \hskip 2mm \text{for each} \hskip 2mm i \in \tI \}.$$
\end{defn}

\begin{thm}[Kashiwara, Naito-Sagaki]\label{energy}
We have a classical crystal $\mathbb B ( \varpi_i )$ corresponding to a $($finite-dimensional$)$ level zero fundamental representation $\mathbb W ( \varpi _i )$ such that $\mathsf{ch} \, \mathbb B ( \varpi_i ) = \mathsf{ch} \, \mathbb W ( \varpi _i )$ for each $i \in \tI$.\\
For each $\lambda = \sum_{i =1}^r m_i \varpi_i \in P_+$, we define the tensor product $($classical$)$ crystal $\mathbb B _{\mathrm{loc}} ( \lambda )$ as
$$\mathbb B _{\mathrm{loc}} ( \lambda ) := \mathbb B ( \varpi_1) ^{\otimes m_1} \otimes \cdots \otimes \mathbb B ( \varpi_r )^{\otimes m_r}.$$
Then, $\mathbb B _{\mathrm{loc}} ( \lambda )$ is a connected crystal and is equipped with a function $D : \mathbb B _{\mathrm{loc}} ( \lambda ) \rightarrow \Z$ with the following conditions:
\begin{enumerate}
\item $D$ is preserved by the action of $\tilde{e}_i$ and $\tilde{f}_i$ for each $i \in \tI$;
\item $D ( b_{0} ) = 0$, where $b_0 \in \mathbb B _{\mathrm{loc}} ( \lambda )$ is the unique element such that $\mathrm{wt} \, b_0 = \lambda$;
\item we have $D ( \tilde{e}_0 b ) = D ( b ) - 1$ for each $b \in \mathbb B_{\mathrm{loc}} ( \lambda )$ such that $\varepsilon_0 ( b ) \ge 2$.
\end{enumerate}
\end{thm}

\begin{proof}
The definition of $\mathbb B ( \varpi_i )$ is due to Kashiwara \cite[Theorem 5.17]{Kas02}. The character comparison follows from the works of Naito-Sagaki \cite{NS08} and Chari-Ion \cite{CI15} (cf. \cite[Theorem 1.6]{Kat16}). The function $D$ is studied in \cite[\S 3]{NS08} under the name of degree function and its relation to the energy statistic is in \cite[Theorem 4.5]{LNSSS2}. The first two properties of $D$ is in \cite[Lemma 3.2.1]{NS08}, while the third property of $D$ also follow from \cite[(3.2.1)]{NS08} in view of \cite[Lemma 2.2.11]{NS08}.
\end{proof}

\begin{thm}[Kashiwara et al. \cite{KMS2, KMS95}, see Hong-Kang \cite{HK02} \S 10]\label{vertex}
Let $\mu \in P_+$. We have an isomorphism
$$\mathbb B ( k \Lambda_0) \otimes \mathbb B _{\mathrm{loc}} ( \mu ) \cong \bigoplus_{b \in \mathcal P ( \mathbb B _{\mathrm{loc}} ( \mu ), k \Lambda_0 )} \mathbb B ( \mathrm{wt} \, b + k \Lambda_0 - D ( b ) \delta )$$
of affine crystals.
\end{thm}

\begin{proof}
Let $\mathbb W$ be the tensor product representation of $U_q ( \widehat{\g} )$ corresponding to $\mathbb B _{\mathrm{loc}} ( \mu )$ borrowed from Theorem \ref{energy}. Being a finite-dimensional integrable representation of the quantum group of $\widehat{\g}$, we can apply the argument in \cite[\S 10.4]{HK02}. By calculating using the lower global basis, we have
$$\sum_{b \in \mathcal P ( \mathbb B _{\mathrm{loc}} ( \mu ), k \Lambda_0 )} e ^{\mathrm{wt} \, b} = \mathsf{ch} \, \{v \in \mathbb W \mid \mathsf E_0^{k + 1} v = 0, \mathsf E_i v = 0, \hskip 3mm i \in \tI \},$$
where $\mathsf E_i \in U_q ( \widehat{\g} )$ ($i \in \tI_\af$) is $e_i$ in \cite[Definition 3.1.1]{HK02}. Applying \cite[Theorem 10.4.3, Theorem 10.4.4]{HK02}, we deduce a classical crystal morphism
$$\Psi : \bigoplus_{b \in \mathcal P ( \mathbb B _{\mathrm{loc}} ( \mu ), k \Lambda_0 )} \mathbb B ( \mathrm{wt} \, b + k \Lambda_0 ) \longrightarrow \mathbb B ( k \Lambda_0) \otimes \mathbb B _{\mathrm{loc}} ( \mu ).$$
Since each direct summand of the LHS is a connected crystal with a highest weight element (Theorem \ref{Blam}) and the image of the highest weight elements are distinct, it follows that $\Psi$ is an embedding of crystals. Therefore, the set $\mathbb B ( k \Lambda_0 ) \otimes \mathbb B _{\mathrm{loc}} ( \mu ) \backslash \mathrm{Im}\, \Psi$ is a classical crystal. As all the highest weight elements of $\mathbb B ( k \Lambda_0) \otimes \mathbb B _{\mathrm{loc}} ( \mu )$ are contained in $\mathrm{Im} \, \Psi$, it does not contain a highest weight element. A successive application of $\{\widetilde{e}_i\}_{i \in \tI_\af}$ sends an arbitrary element of $\mathbb B ( k \Lambda_0 ) \otimes \mathbb B _{\mathrm{loc}} ( \mu )$ to $b_{k \Lambda_0} \otimes \mathbb B _{\mathrm{loc}} ( \mu )$. The set $b_{k \Lambda_0} \otimes \mathbb B _{\mathrm{loc}} ( \mu )$ is stable under the action of $\{\widetilde{e}_i\}_{i \in \tI_\af}$. Therefore, it suffices to prove that applying sufficiently many $\{\widetilde{e}_i\}_{i \in \tI_\af}$ annihilates every element of $b_{k \Lambda_0} \otimes \mathbb B _{\mathrm{loc}} ( \mu )$ in order to prove $\mathbb B ( k \Lambda_0 ) \otimes \mathbb B _{\mathrm{loc}} ( \mu ) = \mathrm{Im}\, \Psi$.

By the definition of the tensor product action and Theorem \ref{energy}, we deduce that the value of $D$ decreases (by one) when we apply $\widetilde{e}_0$ to $b_{k \Lambda_0} \otimes \mathbb B _{\mathrm{loc}} ( \mu )$. Since $b_{k \Lambda_0} \otimes \mathbb B _{\mathrm{loc}} ( \mu )$ is a finite set, it follows that applying $\widetilde{e}_0$ (and other $\{\widetilde{e}_i\}_{i \in \tI}$) sufficiently many times annihilates the whole of $b_{k \Lambda_0} \otimes \mathbb B _{\mathrm{loc}} ( \mu )$. Therefore, we deduce that $\Psi$ is a bijection.

By the same argument, we conclude that $- D ( \bullet )\delta$ equips $b_{k \Lambda_0} \otimes \mathbb B _{\mathrm{loc}} ( \mu )$ with $\widetilde{P}$-valued weights. This makes $\Psi$ into an isomorphism of affine crystals as required.  
\end{proof}

\begin{rem}
Theorem \ref{energy} only states that we can equip $\mathbb B ( k \Lambda_0) \otimes \mathbb B _{\mathrm{loc}} ( \mu )$ a structure of affine highest weight crystals such that $b_{k \Lambda} \otimes \mathbb B _{\mathrm{loc}} ( \mu )$ contains all the highest weight vectors and the $D$-function gives the affine weights of the tensor product. In particular, we have
$$\sum_{b \in \mathcal P ( \mathbb B _{\mathrm{loc}} ( \mu ), k \Lambda_0 )} \mathsf{ch} \, \mathbb B ( \mathrm{wt} \, b + k \Lambda_0 - D ( b ) \delta ) \neq \sum_{b' \otimes b \in \mathbb B ( k \Lambda_0 ) \otimes \mathbb B _{\mathrm{loc}} ( \mu )} q^{D ( b ) \delta} e^{\mathrm{wt} \, ( b' ) + \mathrm{wt} \, ( b )}$$
in general. Note that $\mathrm{wt} \, ( b ) \in P$ (as $\mathbb B _{\mathrm{loc}} ( \mu )$ is a classical crystal), while $\mathrm{wt} \, ( b' ) - k \Lambda_0 \in P \times \Z \delta$ (as $\mathbb B ( \mathrm{wt} \, b + k \Lambda_0 - D ( b ) \delta )$ is an affine crystal).
\end{rem}

Using Theorem \ref{energy}, we define
$$\mathsf{gch} \, \mathbb B_{\mathrm{loc}} ( \lambda ) := \sum_{b \in \mathbb B_{\mathrm{loc}} ( \lambda )} q^{D ( b )} e ^{\mathrm{wt} \, ( b )} \in \Z [\widetilde{P}^0]$$
for each $\lambda \in P_+$.

\begin{defn}[Restricted Kostka polynomials]
For each $\mu, \lambda \in P_+$, we define
$$X_{\mu, \lambda} ( q ) := \sum_{b \in \mathcal P _0 ( \mathbb B_{\mathrm{loc}} ( \mu ), 0 ), \mathrm{wt} \, ( b ) = \lambda} q^{- D ( b )}.$$
For each $\mu \in P_+$ and $\lambda \in P_+^k$, we define
$$X_{\mu, \lambda} ^{(k)} ( q ) := \sum_{b \in \mathcal P ( \mathbb B_{\mathrm{loc}} ( \mu ), k \Lambda_0 ), \mathrm{wt} \, ( b ) = \lambda} q^{- D ( b )}.$$
\end{defn}

\begin{thm}[Lenart-Naito-Sagaki-Schilling-Shimozono \cite{LNSSS2}]\label{LNS3}
For each $\mu, \lambda \in P_+$, we have an equality
$$P_{\mu} ( q, 0 ) = \sum_{b \in \mathbb B _{\mathrm{loc}} ( \mu )} q^{- D ( b )} e^{\mathrm{wt} \, ( b )} = \sum_{\lambda \in P_+} X_{\mu, \lambda} ( q ) \cdot \mathsf{ch} \, V( \lambda ) = \overline{\mathsf{gch} \, \mathbb B_{\mathrm{loc}} ( \mu )},$$
where $P_{\mu} ( q,t )$ is the Macdonald polynomial.\hfill $\Box$
\end{thm}

The following result is a straight-forward extension of a result due to Okado \cite{O} (which in turn uses \cite{HKKOTY}) for type $\mathsf{A}$ (see also Schilling-Shimozono \cite[\S 3.6]{SS00}).

\begin{thm}\label{P-char}
Let $\mu \in P_+$ and $\lambda \in P_+^k$. We have
$$X_{\mu, \lambda} ^{(k)} ( q ) = \sum_{w \in W_\af, \overline{w \circ \lambda} \in P_+} (-1)^{\ell ( w )} q ^{\left< d, \la - w \circ_k \la \right>} X_{\mu, \overline{w \circ_k \lambda}} ( q ).$$
\end{thm}

\begin{proof}

We set $\mathbb B_{\mathrm{loc}} ( \mu; \la ) := \{ b \in \mathbb B_{\mathrm{loc}} ( \mu ) \mid \mathrm{wt} \, ( b ) = \la \}$ for each $\la \in P$. In view of (\ref{WCF}) and the $\Z$-linearity of $\chi$, we deduce
\begin{align}\nonumber
\chi ( e ^{k \Lambda_0} \cdot \,  & \mathsf{gch} \, \mathbb B_{\mathrm{loc}} ( \mu ) ) = \sum_{b \in \mathbb B_{\mathrm{loc}} ( \mu )} \chi ( e ^{\mathrm{wt} \, ( b ) - D ( b ) \delta + k \Lambda_0})\\\nonumber
& = \sum_{\lambda \in P} \sum_{w \in W_\af, \overline{w \circ_k \lambda} \in P_+^k} \sum_{b \in \mathbb B_{\mathrm{loc}} ( \mu ; \la )} \chi ( e^{\la + k \Lambda_0 - D ( b ) \delta} )\\\nonumber
& = \sum_{\lambda \in P} \sum_{w \in W_\af, \la_+ = \overline{w \circ_k \lambda} \in P_+^k} \sum_{b \in \mathbb B_{\mathrm{loc}} ( \mu ; \la )} (-1)^{\ell ( w )} q ^{\left< d, \la - \la_+ \right>} \chi ( e^{\la_+ + k \Lambda_0 - D ( b ) \delta} )\\
& = \sum_{\lambda \in P_+} \sum_{w \in W \backslash W_\af, \la_+ = \overline{w \circ_k \lambda} \in P_+^k} (-1)^{\ell ( w )} q ^{\left< d, \la - \la_+ \right>} X_{\mu, \lambda} ( q^{-1} ) \chi ( e^{\la_+ + k \Lambda_0} ).\label{eq1}
\end{align}
Here we used the fact that the reflection by $\circ_k$ is compatible with the function $D$ by Theorem \ref{energy} 1) and 3) in order to derive the third equality. The fourth equality follows from Theorem \ref{LNS3} and $\chi ( e ^{k \Lambda_0} \cdot \mathsf{ch} \, V ( \la ) ) = \chi ( e^{\la + k \Lambda_0} )$ for each $\la \in P_+$ and $k \in \Z_{> 0}$.

The tensor product crystal $\mathbb B ( k \Lambda_0 )\otimes \mathbb B_{\mathrm{loc}} ( \mu )$ is a classical crystal generated by its highest weight elements.

The set $b_{k \Lambda_0} \otimes \mathbb B_{\mathrm{loc}} ( \mu )$ decomposes into the disjoint union
\begin{equation}
b_{k \Lambda_0} \otimes \mathbb B_{\mathrm{loc}} ( \mu ) = \bigsqcup_{t \ge 1} b_{k \Lambda_0} \otimes \mathbb B_{\mathrm{loc}} ( \mu )^t \subset \bigsqcup_{t \ge 1} \mathbb B ( \Lambda^t ) \hskip 5mm \Lambda^t \in \widetilde{P}_+^k\label{disj}
\end{equation}
that respects the tensor product decomposition in Theorem \ref{vertex}. Here we warn that we equip $\mathbb B ( \Lambda^t )$ a structure of affine crystals. Since $b_{k \Lambda_0} \otimes \mathbb B_{\mathrm{loc}} ( \mu )^t$ is stable under the action of $\{\widetilde{e}_i\}_{i \in \tI_\af}$, the embedding $b_{k \Lambda_0} \otimes \mathbb B_{\mathrm{loc}} ( \mu )^t \subset \mathbb B ( \Lambda^t )$ is stable under the action of $\{\widetilde{e}_i\}_{i \in \tI_\af}$.

For each $w \in W$ and $i \in \tI_\af$, Theorem \ref{Demazure} implies that $\mathbb B ( k \Lambda_0 )_w$ is a disjoint union of $i$-strings or highest weight elements in $i$-strings. Here $\mathbb B_{\mathrm{loc}} (\mu)$ is stable under the action of $\{ \widetilde{e}_i, \widetilde{f}_i\}$.

By a rank one calculation, we deduce that $\mathbb B ( k \Lambda_0 )_w \otimes \mathbb B_{\mathrm{loc}} ( \mu )$ is a disjoint union of $i$-strings or highest weight elements in $i$-strings. Therefore, the set $$\mathbb B ( \Lambda^t )_w := \mathbb B ( k \Lambda_0 )_w \otimes \mathbb B_{\mathrm{loc}} ( \mu ) \cap \mathbb B ( \Lambda^t ) \subset \mathbb B ( \Lambda^t )$$
is a disjoint union of $i$-strings or highest weight elements in $i$-strings. In addition, we have
$$\mathbb B ( \Lambda^t )_{s_iw} = \bigcup_{m \ge 0} \widetilde{f}_i^m \mathbb B ( \Lambda^t )_w$$
for each $i \in \tI_\af$ such that $s_i w > w$ by Theorem \ref{Demazure} and rank one calculation.

In particular, we have
$$\sum_{b \in \mathbb B ( \Lambda^t )_{s_iw}} e^{\mathrm{wt} \, b} = \sum_{b \in \mathbb B ( \Lambda^t )_{w}} D_i ( e^{\mathrm{wt} \, b} )$$
for each $i \in \tI_\af$ such that $s_i w > w$ (here we again warn that the weight here is affine weight).

Applying (\ref{WCFs}), we deduce that
\begin{align}\nonumber
\chi ( e^{k \Lambda_0} \cdot \mathsf{gch} \, \mathbb B_{\mathrm{loc}} (\mu)) & = \sum_{b' \in \mathbb B_{\mathrm{loc}} (\mu)} q^{D(b')}\chi ( e^{k \Lambda_0 + \mathrm{wt} \, b'} )\\\nonumber
& = \sum_{t \ge 1} \sum_{b' \in \mathbb B_{\mathrm{loc}} (\mu)^t} q^{D(b')} \chi ( e^{k \Lambda_0 + \mathrm{wt} \, b'} )\\\nonumber
& = \lim _{w \to \infty} \sum_{t \ge 1} \sum_{b \in \mathbb B ( \Lambda^t )_{w}} \chi ( e^{k \Lambda_0 + \mathrm{wt} \, b} )\\\label{eq2}
& = \sum_{t \ge 1} \chi ( \Lambda^t ) = \sum_{b \in \mathcal P ( \mathbb B_{\mathrm{loc}} (\mu), k \Lambda _0)} q^{D ( b )} \chi ( e^{k \Lambda_0 + \mathrm{wt} \, b}),
\end{align}
where $w \to \infty$ means that we take a limit $\lim_{k \to \infty} s_{i_k} \cdots s_{i_2} s_{i_1}$. Therefore, equating (\ref{eq1}) and (\ref{eq2}) implies the result as required (with $q \mapsto q^{-1}$).
\end{proof}

\begin{cor}\label{comp-mac}
For each $\lambda \in P^k_+$ and $\mu \in P_+$, we have $P_{\mu,\lambda} ^{(k)} ( q ) = X_{\mu,\lambda} ^{(k)} ( q )$.
\end{cor}

\begin{proof}
Compare them using Corollary \ref{altP} and Theorem \ref{P-char} since we have $P_{\lambda, \mu} ( q ) = X_{\lambda, \mu} ( q )$ by Lemma \ref{inf-Kos} and Theorem \ref{LNS3}.
\end{proof}

{\footnotesize
\bibliography{rkref}
\bibliographystyle{hplain}}

\newpage

\appendix

\small{
\def\thesection{Appendix \Alph{section}}
\section{A free field realization proof of Theorem \ref{mult-one} by Ryosuke Kodera\footnote{\texttt{rkodera@math.kyoto-u.ac.jp}, Department of Mathematics, Kyoto University, Oiwake, Kita-Shirakawa, Sakyo Kyoto 606-8502 JAPAN}\footnote{
{\it current address} \texttt{kodera@math.kobe-u.ac.jp} Department of Mathematics, Kobe University, 1-1 Rokkodai, Nada, Kobe Hyogo 657-8501 JAPAN}}
\def\thesection{\Alph{section}}

For general notation, we refer to \S \ref{prelim} and the beginning of \S \ref{wf} in the main body. Let $\g$ be a simple Lie algebra of type $\mathsf{ADE}$. The goal of this appendix is to provide a proof of the following result (without using Theorem \ref{FC-str}) based on the free field realizations.

\begin{thm}[$=$ Theorem \ref{mult-one}]
Let $\varpi \in P_+^1$. The $\tg$-module $L_{1} ( \varpi )$, viewed as a $\g [z]$-module, admits a filtration by $\{ W ( \mu )\}_{\mu \in P_+}$. Moreover, we have
$$( L_{1} ( \varpi ) : W ( \mu ) )_q = \begin{cases} q ^{- \left< d, w ( \varpi + \Lambda_0 ) \right>} & (\mu = \overline{w ( \varpi + \Lambda_0 )}, w \in W_{\af})\\ 0 & (\text{otherwise})\end{cases}$$
for each $\mu \in P_+$.
\end{thm}

Thanks to \cite[(1.25)]{CF13}, it suffices to show:

\begin{prop}\label{App-main}
There exists a filtration of $L_1(\varpi)$ whose adjoint graded quotients yield the inequality
\[
	\mathsf{ch} \ L_1(\varpi) \leq \sum_{\lambda \in P_+ \cap \overline{W_{\mathrm{af}}(\varpi + \Lambda_0)}} q^{\frac{1}{2} \left( (\lambda, \lambda) - (\varpi,\varpi) \right)} \mathsf{ch} \ W(\lambda).
\]
\end{prop}

The rest of this appendix is devoted to the proof of Proposition \ref{App-main}.

We recall the Frenkel-Kac construction of level one integrable representations of $\tg$.
Let $\varpi$ be an element of $P^1_+$ and $L_1(\varpi)$ be the integrable highest weight $\tg$-module with highest weight $\varpi + \Lambda_0$.
\begin{lem}\label{lem:extremal_weights}
We have	$W_{\mathrm{af}}(\varpi+\Lambda_0) = \{ t_{\gamma}(\varpi+\Lambda_0) \mid \gamma \in Q \}$. In particular, extremal weights in $L_1(\varpi)$ are parametrized by $Q$.
\end{lem}
\begin{proof}
The equality follows by \cite[Lemma~12.6]{Kac}.
\end{proof}

\begin{rem}
In this case we also have $\max(\varpi+\Lambda_0) = W_{\mathrm{af}}(\varpi+\Lambda_0)$.
Here $\max(\varpi+\Lambda_0)$ denotes the set of the maximal weights of $L_1(\varpi)$.
\end{rem}

\begin{lem}\label{lem:bar_extremal_weights}
We have $P_+ \cap \overline{W_{\mathrm{af}}(\varpi + \Lambda_0)} = P_+ \cap (\varpi + Q_+)$.
\end{lem}
\begin{proof}
We have $\overline{W_{\mathrm{af}}(\varpi + \Lambda_0)} = \varpi + Q$ by Lemma~\ref{lem:extremal_weights}.
It is well known that for any $\lambda \in P_+$ there exists unique $\varpi' \in P^1_+ \simeq P/Q$ such that $\lambda \geq \varpi'$.
Thus we have $P_+ \cap (\varpi + Q) = P_+ \cap (\varpi + Q_+)$.
\end{proof}

We use the symbol $X(k) = X \otimes \xi^k = X \otimes z^{-k} \in \tg$ for an element $X \in \g$.
Define a Lie subalgebra $\wts$ of $\tg$ to be $\wts = \h[\xi,\xi^{-1}] \oplus \C K$.
Then $\wts$ is a direct sum of the Heisenberg Lie algebra $\ws = \bigoplus_{k \neq 0} (\h \otimes \xi^{k}) \oplus \C K$ and an abelian Lie algebra $\h \otimes 1$.
Put $\ws_{\geq 0} = \bigoplus_{k > 0} (\h \otimes \xi^{k}) \oplus \C K$ and let $\C_1$ be the one-dimensional representation of $\ws_{\geq 0}$ via $\h \otimes \xi^{k} \mapsto 0$ ($k>0$) and $K \mapsto \id$.
Let $F$ be the Fock representation of the Heisenberg Lie algebra $\ws$ defined by the induction
\[
	F = U(\ws) \otimes_{U(\ws_{\geq 0})} \C_{1}.
\]
We denote by $\vac$ the element $1 \otimes 1 \in F$. % and call it the vacuum vector.
Consider $\C[Q]$ the group algebra of $Q$. It has a $\C$-basis $e^{\gamma}$ ($\gamma \in Q$) and the multiplication is given by $e^{\beta}e^{\gamma}=e^{\beta+\gamma}$.
We denote by $e^{\varpi}\C[Q]$ a $\C$-vector space which has a $\C$-basis $e^{\varpi+\gamma}$ ($\gamma \in Q$).
An action of $\h \otimes 1$ on $e^{\varpi}\C[Q]$ is given by $h(0) e^{\varpi+\gamma} = \langle h, \varpi+\gamma \rangle e^{\varpi+\gamma}$ for $h \in \h$.
Then $F \otimes e^{\varpi}\C[Q]$ is naturally a module of $\wts = \ws \oplus (\h \otimes 1)$.
We define a $\Z$-grading on $F \otimes e^{\varpi}\C[Q]$ by
\[
	\deg (\h \otimes \xi^{-k}) = k \ (k \geq 1) \text{ and } \deg e^{\varpi+\gamma} = (\varpi, \gamma) + \dfrac{1}{2}(\gamma, \gamma).
\]
Thus $F \otimes e^{\varpi}\C[Q]$ is extended to a module of $\wts \oplus \C d$ such that $(-d)$ counts the degree.
\begin{rem}
The degree of $e^{\varpi+\gamma}$ is determined such that the $\widetilde{\h}$-weight of $\vac \otimes e^{\varpi + \gamma}$ is $t_{\gamma} (\varpi + \Lambda_0)$.
\end{rem}

We take a certain 2-cocycle $\varepsilon \colon Q \times Q \to \{\pm 1\}$ as in \cite[2.3]{FK80}. For an element $\gamma \in Q$, we define an operator $\widetilde{T}_{\gamma}$ on $e^{\varpi}\C[Q]$ by $\widetilde{T}_{\gamma} e^{\varpi+\beta} = \varepsilon(\gamma,\beta) e^{\varpi + \beta + \gamma}$ ($\beta \in Q$).
By \cite[Proposition~2.2]{FK80}, we can choose root vectors $E_{\alpha} \in \g_{\alpha}$ for $\alpha \in \Delta$ satisfying certain relations, e.g., $[E_{\alpha}, E_{\beta}] = \varepsilon(\beta,\alpha) E_{\alpha+\beta}$ if $\alpha + \beta \in \Delta$.

\begin{thm}[Frenkel-Kac \cite{FK80}]\label{thm:Frenkel-Kac}
\begin{enumerate}
\item The restriction of the level one representation $L_1(\varpi)$ of $\tg$ to $\wts \oplus \C d$ is isomorphic to $F \otimes e^{\varpi}\C[Q]$.
\item The $(\wts \oplus \C d)$-module $F \otimes e^{\varpi}\C[Q]$ is extended to $\tg$ by
\begin{equation*}
	\sum_{k \in \Z} E_{\alpha}(k) u^{-k} \mapsto \exp\left( \sum_{k>0} \dfrac{\alpha^{\vee} (-k)}{k} u^{k} \right) (\widetilde{T}_{\alpha} u^{1+\alpha^{\vee}}) \exp\left( -\sum_{k>0} \dfrac{\alpha^{\vee} (k)}{k} u^{-k} \right)
\end{equation*}
and it is isomorphic to $L_1(\varpi)$ as a $\tg$-module.
\end{enumerate}
\end{thm}
\begin{proof}
The assertion is proved by \cite{FK80} for the case $\varpi=0$.

A proof of (i) for a general $\varpi$ is similar.
We give a sketch.
Let $\bv_{\varpi+\gamma}'$ be an extremal weight vector in $L_1(\varpi)$ of weight $t_{\gamma}(\varpi+\Lambda_0)$.
Then $U(\wts \oplus \C d) \bv_{\varpi+\gamma}'$ is isomorphic to $F \otimes e^{\varpi+\gamma}$ as a module of $\wts \oplus \C d$.
Hence we have an injection $F \otimes e^{\varpi}\C[Q] \to L_1(\varpi)$.
By comparing their characters, we see that they are isomorphic.

A proof of (ii) is the same as \cite{FK80}.
\end{proof}

For each $\gamma \in Q$, we put $\bv_{\varpi+\gamma} = \widetilde{T}_{\gamma} (\vac \otimes e^{\varpi})= \vac \otimes e^{\varpi+\gamma} \in F \otimes e^{\varpi}\C[Q]$. (We note that $\varepsilon(\gamma,0) = 1$.)
We regard $\bv_{\varpi+\gamma}$ as an extremal weight vector in $L_1(\varpi)$ via the Frenkel-Kac construction.
The $\widetilde{\h}$-weight of $\bv_{\varpi+\gamma}$ is
\[
	\varpi +\gamma - \left( (\varpi, \gamma) + \dfrac{1}{2}(\gamma, \gamma) \right) \delta + \Lambda_0
\] by construction.
Hence the following lemma follows.
\begin{lem}\label{lem:extremal}
The vector $\bv_{\varpi+\gamma}$ is an extremal weight vector of weight $t_{\gamma} (\varpi + \Lambda_0)$.
Moreover, $\mathfrak{n} \bv_{\varpi+\gamma} = 0$ if and only if $\varpi+\gamma$ is dominant.
\end{lem}
%\begin{proof}
%The first statement follows by [Frenkel-Kac~Prop~2.3 b and c].
%In [loc.\ cit.], the action of $\widetilde{T}_{\gamma}$ is identified with the braid group action corresponding to the translation $t_{\gamma}$ in $W_{\mathrm{af}}$. (Here we use $Q=Q^{\vee}$.)

%To prove the second statement, we use the realization of the action of $\sum_{k \in \Z} E_{\alpha}(k) u^{-k}$ by the vertex operator
%\[
%	\exp\left( \sum_{k>0} \dfrac{\alpha^{\vee} (-k)}{k} u^{k} \right) (\widetilde{T}_{\alpha} u^{1+\alpha^{\vee}}) \exp\left( -\sum_{k>0} \dfrac{\alpha^{\vee} (k)}{k} u^{-k} \right).
%\]
%The vector $v_{\gamma+\varpi}$ is killed by $\xi\h[\xi]$ by definition.
%Hence we have
%\[
%	\left( \sum_{k \in \Z} E_{\alpha}(k) u^{-k} \right) v_{\gamma+\varpi} = \exp\left( \sum_{k>0} \dfrac{\alpha^{\vee} (-k)}{k} u^{k} \right) \widetilde{T}_{\alpha} u^{1+\langle \alpha^{\vee}, \gamma+\varpi \rangle} v_{\gamma+\varpi}.
%\]
%Then the right-hand side belongs to $u L_1(\varpi) [[u]]$ provided $\gamma+\varpi$ is dominant and $\alpha$ is a positive root.
%This implies that $E_{\alpha} v_{\gamma+\varpi} = 0$ for every $\alpha \in \Delta_+$ if $\gamma+\varpi \in P_+ \cap (\varpi + Q_+)$.
%\end{proof}

The following lemma follows from Lemma~\ref{lem:extremal_weights}.
\begin{lem}\label{lem:extremal2}
	Any extremal weight vector in $L_1(\varpi)$ is of the form $\bv_{\varpi+\gamma}$ $(\gamma \in Q)$ up to scalar.
\end{lem}

We use the Frenkel-Kac construction to prove that $L_1(\varpi)$ has a filtration whose successive quotients are quotients of global Weyl modules. We set
$$\mathrm{gr}_{\la} \, L_1(\varpi) := U(\g[z]) \bv_{\lambda} / \displaystyle\sum_{\stackrel{\mu \in P_+ \cap (\varpi + Q_+),}{\mu > \lambda}} U(\g[z]) \bv_{\mu} \hskip 5mm \la \in P_+ \cap (\varpi + Q_+).$$
\begin{prop}\label{prop:surjection}
Let $\lambda$ be an element of $P_+ \cap (\varpi + Q_+)$.
The image $\bar{\bv}_{\lambda}$ of $\bv_{\lambda}$ in $\mathrm{gr}_{\la} \, L_1(\varpi)$ satisfies
\[
	h(0) \bar{\bv}_{\lambda} = \langle h, \lambda \rangle \bar{\bv}_{\lambda} \ (h \in \h) \ \text{ and } \ \mathfrak{n} [z] \bar{\bv}_{\lambda} = 0.
\]
Hence we have a surjective morphism of degree $\frac{1}{2} \left( (\lambda, \lambda) - (\varpi,\varpi) \right)$ from the global Weyl module $W(\lambda)$ to $\mathrm{gr}_{\la} \, L_1(\varpi)$.
\end{prop}
\begin{proof}
The relation $h(0) \bar{\bv}_{\lambda} = \langle h, \lambda \rangle \bar{\bv}_{\lambda}$ follows since the $\h$-weight of $\bv_{\lambda}$ is $\lambda$ by construction.

Let $\alpha \in \Delta_+$.
Then for $k \geq 0$, we have 
\[
E_{\alpha}(-k) \bv_{\lambda} \in F \otimes e^{\lambda + \alpha} = U(\xi^{-1}\h[\xi^{-1}]) (\vac \otimes e^{\lambda + \alpha})
\]
by Theorem~\ref{thm:Frenkel-Kac}. 
Here $\vac \otimes e^{\lambda + \alpha} = \bv_{\lambda + \alpha}$ is an extremal weight vector in $L_1(\varpi)$.
The $\g$-submodule $U(\g) \bv_{\lambda+\alpha}$ is finite-dimensional and simple.
Let $\mu \in P_+$ be the highest weight of this module.
Then $U(\g) \bv_{\lambda+\alpha}$ contains $\bv_{\mu}$ as its highest weight vector by Lemma~\ref{lem:extremal} and \ref{lem:extremal2}.
Hence we see that $\bv_{\lambda+\alpha} \in U(\mathfrak{n}^-) \bv_{\mu}$ and $\lambda < \lambda + \alpha \leq \mu$.
This implies that
\[
	E_{\alpha}(-k) \bv_{\lambda} \in U(\xi^{-1}\h[\xi^{-1}]) U(\mathfrak{n}^-) \bv_{\mu}.
\]
and completes the proof.
\end{proof}

\begin{proof}[Proof of Proposition \ref{App-main}]
The filtration is constructed as above. The inequality
\[
	\mathsf{ch} \ L_1(\varpi) \leq \sum_{\lambda \in P_+ \cap (\varpi + Q_+)} q^{\frac{1}{2} \left( (\lambda, \lambda) - (\varpi,\varpi) \right)} \mathsf{ch} \ W(\lambda)
\]
follows from Proposition~\ref{prop:surjection}.
The summation in the right-hand side is over $P_+ \cap \overline{W_{\mathrm{af}}(\varpi + \Lambda_0)}$ by Lemma~\ref{lem:bar_extremal_weights}.
\end{proof}

}

{\small{\bf Acknowledgments:}
R.K. thanks Sergey Loktev for the explaining his idea to use the free field realization in the proof of Theorem \ref{mult-one}.}

\end{document}